\documentclass[11pt,leqno]{article}
\usepackage{amsthm,amsfonts,amssymb,amsmath,oldgerm,cases,upgreek}
\usepackage{epsfig}
\numberwithin{equation}{section}
\usepackage[thinlines]{easybmat}
\usepackage{enumitem}
\usepackage{enumitem}
\usepackage{tikz,tkz-tab}
\def\eps{\varepsilon }

\newcommand\R{\mathbb R}
\newcommand{\Rn}{\mathbb{R}^{n}}

\newcommand{\lver}{\left\lvert}

\newcommand{\rver}{\right\rvert}

\def\eps{\varepsilon}

\newcommand\br{\begin{remark}}
\newcommand\er{\end{remark}}
\newcommand\bp{\begin{pmatrix}}
\newcommand\ep{\end{pmatrix}}
\newcommand\be{\begin{equation}}
\newcommand\ee{\end{equation}}
\newcommand\ba{\begin{equation}\begin{aligned}}
\newcommand\ea{\end{aligned}\end{equation}}

\newcommand{\bap}{\begin{app}}
\newcommand{\eap}{\end{app}}
\newcommand{\begs}{\begin{exams}}
\newcommand{\eegs}{\end{exams}}
\newcommand{\beg}{\begin{example}}
\newcommand{\eeg}{\end{exaplem}}
\newcommand{\bpr}{\begin{proposition}}
\newcommand{\epr}{\end{proposition}}
\newcommand{\bt}{\begin{theorem}}
\newcommand{\et}{\end{theorem}}
\newcommand{\bc}{\begin{corollary}}
\newcommand{\ec}{\end{corollary}}
\newcommand{\bl}{\begin{lemma}}
\newcommand{\el}{\end{lemma}}
\newcommand{\bd}{\begin{definition}}
\newcommand{\ed}{\end{definition}}
\newcommand{\brs}{\begin{remarks}}
\newcommand{\ers}{\end{remarks}}

\newtheorem{theo}{Theorem}[section]
\newtheorem{prop}[theo]{Proposition}

\newtheorem{exams}[theo]{Examples}

\numberwithin{equation}{section}
\newcommand{\TT}{{\mathbb T}}

\newcommand{\const}{\text{\rm constant}}
\newcommand{\Id}{{\rm Id }}

\newcommand{\blockdiag}{{\rm blockdiag }}

\newcommand{\sgn}{\text{\rm sgn}}
\newtheorem{theorem}{Theorem}[section]
\newtheorem{proposition}[theorem]{Proposition}
\newtheorem{corollary}[theorem]{Corollary}
\newtheorem{lemma}[theorem]{Lemma}
\newtheorem{definition}[theorem]{Definition}

\newtheorem{example}[theorem]{Example}
\newtheorem{remark}[theorem]{Remark}

\title{
 Existence and behavior of steady solutions on an interval for general hyperbolic-parabolic systems of conservation laws
}

\author{\sc \small
Benjamin Melinand\thanks{Institut de Recherche Mathématique avancée, UMR 7501, Université de Strasbourg et CNRS, 7 rue René Descartes, 67000
Strasbourg, France;
melinand@unistra.fr
}
and
Kevin Zumbrun\thanks{Indiana University, Bloomington, IN 47405;
kzumbrun@iu.edu: Research of K.Z. was partially supported
under NSF grants no. DMS-1700279, DMS-2154387,and DMS-2206105.
 }}
\begin{document}

\maketitle

%%%%%%%%%%%%%%%%%%%%%%%%%%%%%%%%%%%%%%%%%%%%%%%%%%%%%%%%%%%%%%%%%%%%%%%%%%%%%%%%%%%%%%%%%%%%%

\begin{center}
\end{center}
 
%%%%%%%%%%%%%%%%%%%%%%%%%%%%%%%%%%%%%%%%%%%%%%%%%%%%%%%%%%%%%%%%%%%%%%%%%%%%%%%%%%%%%%%%%%%%%
\begin{abstract}
	We study the inflow-outflow boundary value problem on an interval, the analog of
	the 1D shock tube problem for gas dynamics, for general systems of hyperbolic-parabolic
	conservation laws.
	In a first set of investigations, we study existence, uniqueness, and stability,
	showing in particular local existence, uniqueness, and
	stability of small amplitude solutions for general symmetrizable systems.
	In a second set of investigations, we investigate structure and behavior in the small- and large-viscosity
	limits. 
	
	A phenomenon of particular interest is the generic appearance 
	of characteristic boundary layers in the inviscid limit,
	arising from noncharacteristic data for the viscous problem, 
	even of arbitrarily small amplitude.  This induces an interesting new type of ``transcharacteristic'' 
	hyperbolic boundary condition governing the formal inviscid limit.
\end{abstract}

%%%%%%%%%%%%%%%%%%%%%%%%%
%%%%%%%%%%%%%%%%%%%%%%%%%
\section{Introduction}\label{s:intro}
%%%%%%%%%%%%%%%%%%%%%%%%%%%%%%%%%%%%%%%%%%%%%%%%%
In this paper, inspired by recent results in \cite{MelZ,paper1} on the ``1D shock tube problem''
for gas dynamics, i.e., steady flow on a bounded interval, with noncharacteristic inflow/outflow
boundary conditions, we here begin the systematic study of the ``generalized 1D shock problem''
for arbitrary systems of hyperbolic-parabolic conservation laws.

Our goals are two-fold: first, to add context and larger foundation to the somewhat special
analyses of \cite{MelZ,paper1}, and, second, to introduce what seems to be a family
of new and interesting types of hyperbolic and hyperbolic-parabolic problems for general conservation laws.

\subsection{Equations and assumptions}\label{s:eqs}
Following \cite{SZ,Z1}, we consider steady solutions of general viscous conservation laws:
\begin{equation}\label{conservation_law}
	\partial_{t}f^0(U) + f(U)_{x} = \left(B(U) U_{x} \right)_{x} \text{ , } 0<x<1 \text{ , } t>0,
\end{equation}
together with the boundary conditions
\begin{equation}\label{BC2}
U(0) = U_{0} = \begin{pmatrix} U_{0I} \\ U_{0II} \end{pmatrix} \text{  and  } U_{II}(1) = U_{1II}.
\end{equation}
In this paper we assume that  $U\in \mathcal{U}\subset \R^n$,
\begin{equation*}
U = \begin{pmatrix} U_{I} \\ U_{II} \end{pmatrix}, \, 
f^0 = \begin{pmatrix} f^0_{I} \\ f^0_{II} \end{pmatrix}, \,
f = \begin{pmatrix} f_{I} \\ f_{II} \end{pmatrix} \,
	\in \R^{r} \times \R^{n-r} \text{ , } 
	B = \begin{pmatrix} 0 & 0 \\ 0 & B_{22} \end{pmatrix} \text{ , } B_{22} \in M_{n-r}(\R),
\end{equation*}
and $df^0$ is invertible and lower block triangular, without loss of generality
\begin{equation*}
	df^0(U) = \begin{pmatrix} \Id_r & 0\\ (df^0_{II})_I& (df^0_{II})_{II} \end{pmatrix}
\end{equation*}
with $\det ((df^0_{II})_{II}) >0$. We also make the following assumptions:

\medskip
\medskip
\noindent (\textbf{H0}) $f$ and $B$ are smooth.

\medskip
\medskip
\noindent (\textbf{H1}) The eigenvalues of $(df^0_{II})_{II}^{-1}(U) B_{22}(U) + ((df^0_{II})_{II}^{-1}(U) B_{22}(U))^{T}$ are positive for any $U$.

\medskip
\medskip
\noindent (\textbf{H2}) 
The eigenvalues of $(df_{I})_{I}(U)$ are real and positive for any $U$. 

\medskip
\medskip

\noindent Condition (\textbf{H1}) corresponds to strict parabolicity of \eqref{conservation_law} with
respect to $U_{II}$, consistent with Dirichlet boundary conditions on $U_{II}(0), U_{II}(1)$. 
Note that $B_{22}(U)$ is necessary invertible and $\det B_{22}(U)>0$ for any $U$ (since $\det ((df^0_{II})_{II}) >0$).
The first part of Condition (\textbf{H2}) corresponds to hyperbolicity
of \eqref{conservation_law} with respect to variable $U_I$.
The second part of Condition (\textbf{H2})
means that the flow moves from the left to the right, so that the hyperbolic part of the flow, $U_I$,
is completely entering the domain at the left boundary,
consistent with Dirichlet boundary conditions on $U_I(0)$.

In the following, we denote 
$A^0(U)=df^0(U)$, $A(U)=df(U)$, with
\begin{equation*}
A^0 = \begin{pmatrix} \Id_r & 0 \\ A^0_{21} & A^0_{22} \end{pmatrix},\quad 
A = \begin{pmatrix} A_{11} & A_{12} \\ A_{21} & A_{22} \end{pmatrix} \text{ , } A_{11} \in M_{r}(\R).
\end{equation*}

\subsection{Main results}\label{s:main}

In Section \ref{s:existence}, we categorize essentially completely existence, uniqueness and stability of
small-amplitude steady solutions for general viscous conservation laws \eqref{conservation_law} under steadily
increasing assumptions, culminating with the physically natural assumption of existence of a convex
entropy.  We discuss afterward in a more partial and speculative way the issues of global existence,
uniqueness, and stability, that is, the situation for {\it large-amplitude solutions}.
In particular, we outline a strategy based on entropy dissipation and Brouwer degree for 
global uniqueness of constant solutions and global existence of large-amplitude systems, generalizing
and illuminating the one introduced in \cite{paper1} in the context of gas dynamics.

In Section \ref{s:asymptoticI}, we give a complete study of structure and behavior of steady solutions in the scalar case in the
small viscosity limit. 

In Section \ref{s:asymptoticII}, we turn to the study of structure and behavior of steady solutions, in the
small- and large-viscosity limits. 
Even for the simple case of isentropic gas dynamics, there is a rich ``zoo'' of possible solution patterns
in the inviscid limit similar to what we observe in Section \ref{s:asymptoticI}, featuring standing shocks, left and right boundary layers, and, the most novel feature,
a new type of characteristic boundary layer appearing not as a boundary case, but {\it generically} in
``rarefying,'' or ``expansive'' solutions.
%TODO: mention these words later on in appropriate section.
These induce in the formal inviscid limit a new type of ``transcharacteristic'' boundary condition
accomodating perturbations on either side of a characteristic state, hence requiring different numbers
of boundary conditions; see Section \ref{s:hypdesc}.
This phenomenon occurs for {\it any} rarefying steady solution, even of arbitrarily
small amplitude, hence must be dealt with in order to produce a hyperbolic description of dynamics.

\subsection{Discussion and open problems}\label{s:disc}
A very interesting finding of \cite{paper1} is that nonuniqueness may hold for systems with convex entropy,
in particular even for gas dynamics with an artificially devised equation of state.
Yet, it appears (numerically) to hold for the standard polytropic equation of state.
A very interesting open problem is to verify this analytically.
Likewise, the proof or disproof of existence for more general equations of state, or more general equations such
as magnetohydrodynamics (MHD) are important open problems.
Existence of large-amplitude multi-D solutions is a challenging more long term goal, even for gas dynamics,
that appears to require further ideas for its resolutions.

As regards structure and asymptotic behavior, a very interesting open problem is to determine the possible
feasible hyperbolic configurations for general systems, similarly as done here for isentropic gas dynamics.
The understanding of hyperbolic behavior for the limiting zero-viscosity dynamics is another challenging direction.
The handling of characteristic boundary layers, and resulting transcharacteristic type boundary conditions is
a particularly intriguing piece of this puzzle.
Similarly, the understanding of large-viscosity behavior for general systems is another interesting direction for 
study.

The most intriguing open problem is one that has not been addressed at all here, namely structure of the 
corresponding {\it multi-D solutions} in the inviscid limit.
This could be thought of as a combination of compressible Pouseille flow and the 1-D shock problem;
the results should be very interesting indeed.
In particuar, even for small data, for which multi-D existence and uniqueness are known 
\cite{KK}, the question of structure seems not to have yet been addressed.
An especially interesting question seems to be what is the role in multi-d of the characteristic boundary layer
configurations we have seen here in 1-D.
Presumably such configurations are there in the construction of \cite{KK}, but their
numerical and asymptotic description is still wanting.

%TODO: suspect also that size of allowable data shrinks with viscosity in kellog results! Check this!

\medskip

{\bf Acknowledgement.} We thank Blake Barker for many helpful discussions during the closely related
project \cite{paper1}, and for his generous aid both numerical and analytical.

\section{Existence, uniqueness and spectral stability}\label{s:existence}

\subsection{The linear case}\label{s:linear}
We assume in this part that $df$ and $B$ are both constant. We have the following proposition.

\begin{prop}\label{steadyprop}
If $df$ and $B$ are both constant and conditions (\textbf{H0})-(\textbf{H2}) are satisfied, Problem \eqref{conservation_law} has a unique steady state that satisfies the boundary condition \eqref{BC2} if and only if 
\begin{equation}\label{speccond}
\sigma \left(B_{22}^{-1} \left(A_{22} - A_{21} A_{11}^{-1} A_{12} \right)\right) \cap 2i\pi \mathbb{Z} \backslash \{ 0 \} = \emptyset.
\end{equation}
\end{prop}

\begin{remark}
	Condition \eqref{speccond} is a compatibility condition between the parabolic and the hyperbolic part. A similar condition was assumed for the study of quasilinear noncharacteristic boundary layers (on the half-line) in, for instance, \cite[Lemma 5.1.3]{M1} or \cite{M2}. For example, the following system does not satisfy Condition \eqref{speccond} 
	$$
	U_{t} + \bp 0 & 2\pi\\ -2\pi & 0 \ep U_{x} = U_{xx}  \text{ , } 0<x<1,
	$$
	and any constant state $\hat{U}$ must satisfy $\hat{U}(0)=\hat{U}(1)$.
\end{remark}

\begin{proof}[Proof of Proposition \ref{steadyprop}]
We can rewrite the problem as
\begin{equation*}
\left\{
\begin{array}{l}
A_{11} U_{I}' + A_{12} U_{II}'=0,\\
A_{21} U_{I}' + A_{22} U_{II}'=B_{22} U_{II}''.
\end{array}
\right.
\end{equation*}
Then, integrating, we get 
\begin{equation*}
\left\{
\begin{array}{l}
A_{11} U_{I} + A_{12} U_{II}=C_{1},\\
A_{21} U_{I} + A_{22} U_{II} + C_{2} =B_{22} U_{II}'
\end{array}
\right.
\end{equation*}
where $C_{1}=A_{11} U_{I0} + A_{12} U_{II0}$ and $C_{2}$ is a constant that has to be determined. 
	Then, since $A_{11}$ is invertible, we obtain
\begin{equation}\label{varprof}
\left\{
\begin{array}{l}
U_{I}' = A_{11}^{-1} C_{1} - A_{11}^{-1} A_{12} U_{II},\\
U_{II}' = B_{22}^{-1} (A_{22} - A_{21} A_{11}^{-1} A_{12}) U_{II} + B_{22}^{-1}(C_{2}+A_{21}A_{11}^{-1}C_{1}).
\end{array}
\right.
\end{equation}
Denoting $\tilde{A}=B_{22}^{-1} (A_{22} - A_{21} A_{11}^{-1} A_{12})$ and $\tilde{C}=B_{22}^{-1}(C_{2}+A_{21}A_{11}^{-1}C_{1})$, we solve
\begin{equation*}
\left\{
\begin{array}{l}
U_{II}' = \tilde{A} U_{II} + \tilde{C}\\
U_{II}(0)=U_{II0}.
\end{array}
\right.
\end{equation*}
We decompose $\tilde{A}$ as $\tilde{A} = P^{-1} \begin{pmatrix} F_{1} & 0 \\ 0 & F_{2} \end{pmatrix} P$ where $P$ and $F_{2}$ are invertible and $F_{1}$ is strictly upper triangular.
	For example, $\blockdiag\{F_1,F_2\}$ could be taken to be a Jordan form for $\tilde A$, 
	with $F_1$ the $0$-generalized eigenspace part. We get that
\begin{equation*}
U_{II}(1) = e^{\tilde{A}} U_{II}(0)+P^{-1} \begin{pmatrix} \int_{0}^{1} e^{sF_{1}} ds & 0 \\ 0 & F_{2}^{-1} \left(e^{F_{2}} - I \right) \end{pmatrix} P \tilde{C}.
\end{equation*}
Thus, we see that the map $C_{2} \to U_{II}(1)$ is invertible if and only if $\sigma (\tilde{A}) \cap 2i\pi \mathbb{Z} \backslash \{ 0 \}= \emptyset$.
\end{proof}

\subsection{Almost-constant steady states}

We next study the existence of steady states for system \eqref{conservation_law}-\eqref{BC2} for data lying near a constant state satisfying \eqref{speccond}. We seek solutions $\widehat{U}$ of
\begin{equation}\label{steady_conservation_law}
(f(\widehat{U}))_{x} = \left(B(\widehat{U}) \widehat{U}_{x} \right)_{x} \text{ , } 0< x < 1 \text{ , } \widehat{U}(0) = U_{0} \text{  ,  } \widehat{U}_{II}(1) = U_{1II}.
\end{equation}
Similarly as done for gas dynamics in \cite{MelZ,paper1}, for $U_{0}$ and $U_{1II}$ fixed, we define the map
\begin{equation*}
\Psi : C_{2} \in \R^{n-r}  \to U_{II}(1)-U_{1II} \in \mathbb{R}^{n-r}
\end{equation*}
such that the maximal solution $U$ of the ODE
\be\label{ode_general}
B_{22}(U) U_{II}' = f_{II}(U) - f_{II}(U_{0}) +  B_{22}(U_{0}) C_{2} \text{ , } U_{II}(0)=U_{0II} \text{  ,  } f_{I}(U)=f_{I}(U_{0})
\ee
is defined on $[0,1]$. Thus, profiles are equivalent to roots $C_{2}$ of $\Psi$.

\bpr\label{existence_almost_steady}
Let $U_{0} \in \Rn$ and assume conditions (\textbf{H0})-(\textbf{H2}) are satisfied. Assume that Condition \eqref{speccond} is satisfied for $A=df(U_{0})$ and $B=B(U_{0})$. There exists $\delta>0$ and $\eps>0$ such that for any $U_{1II}$ with $\lver U_{0II}-U_{1II} \rver \leq \delta$, there exists a unique solution $\hat{U}$ of \eqref{steady_conservation_law} satisfying
\begin{equation*}
\lver \hat{U}_{II}'(0) \rver\leq \eps.
\end{equation*}
Moreover, the solution is nondegenerate: i.e., corresponds to a nondegenerate root of $\Psi$.
\epr

\br 
We do not claim that for $U_{1II}$ close enough to $U_{0II}$, there exists a unique solution of \eqref{steady_conservation_law}. The previous theorem only gives local uniqueness, 
%ADDED kz2020
even for small data.
\er 

\begin{proof}[Proof of Proposition \ref{existence_almost_steady}]
	Let us fix $U_{0} \in \Rn$, taking without loss of generality $f_I(U_0)=0$. 
	First, we notice that for $U_{1II}=U_{0II}$, $\hat{U} \equiv U_{0}$ is a solution of \eqref{steady_conservation_law}. Then, by continuous dependence on parameters on the ODE, (\textbf{H2}) and the implicit function theorem on the constraint $f_{I}(U_{I},U_{II})=0$, for $(U_{1II},C_{2})$ close enough to $(U_{0II},0)$, one can express $U_{I}$ as a function of $U_{II}$ and the maximal solution $U$ of \eqref{ode_general} is defined on $[0,1]$. Therefore, we can define the map $\Psi$ on a neighborhood $\mathcal{V}$ of $0$ in $\R^{n-r}$. The function $\Psi$ is $\mathcal{C}^{1}$ on this domain and $\Psi(0)=0$. Then, for any $D \in \mathbb{R}^{n-r}$, $d \Psi(0) \cdot D = V_{II}(1)$, with $V_{II}$ solving the variational equation
\begin{equation*}
B_{22}(U_{0}) V_{II}' = (A_{22} - A_{21} A_{11}^{-1} A_{12})(U_{0}) V_{II} + C_{2} + (A_{21}A_{11}^{-1})(U_{0})C_{1} \text{  ,  } V_{II}(0)=0,
\end{equation*} 
with $C_{1} = A_{11}(U_{0}) U_{I0} + A_{12}(U_{0}) U_{II0}$. As in the proof of Proposition \ref{steadyprop}, we can solve this ODE and, using  Condition \eqref{speccond}, we obtain that $d \Psi(0)$ is invertible. The result then
follows from the implicit function theorem, as does nondegeneracy.
\end{proof}

\subsubsection{The Evans function}\label{s:evans}
Linearizing \eqref{conservation_law} around a steady solution $\hat U$, we obtain eigenvalue equations
\begin{equation}\label{eval}
	\lambda A^0 V + (AV)_{x} = \left(B V_{x} \right)_{x} \text{ , } 0<x<1  \text{ , } t>0,
\end{equation}
$A^0$, $A$, $B$ depending on $x$, with homogeneous boundary conditions $U(0)=0$, $U_{II}(1)=0$.
In particular, $A^0_{11}=I_{r}$, $A^0_{12}=0$,
$A_{11}=df_I(\hat U)$ is invertible, $A_{12}=df_{II}(\hat U)$, and $B=B(\hat U)$ with $B_{22}$ invertible.
Thus, denoting $d/dx$ as $'$, we may rewrite \eqref{eval} as 
a first-order system 
%$\mathcal{Z}'=\mathcal{A}(x)\mathcal {Z}$, where
\ba\label{firstorder}
\bp A_{11}V_I \\ V_{II} \\ B_{22} V_{II}'\ep'= \bp 
- \lambda A_{11}^{-1} & - A_{12}' & - A_{12}B_{22}^{-1}\\
0 & 0 & B_{22}^{-1}\\
 \lambda (A^0_{21} - A_{21} A_{11}^{-1}   )A_{11}^{-1} + \alpha & 
\lambda A_{22}^{0}+\beta
& \gamma \ep \bp A_{11}V_I \\ V_{II} \\ B_{22} V_{II}'\ep
\ea
with
\[
\alpha= A_{21}' A_{11}^{-1} + A_{21}(A_{11}^{-1})', \beta= A_{22}' - A_{21}A_{11}^{-1}A_{12}', \gamma= (A_{22} -  A_{21}A_{11}^{-1} A_{12})B_{22}^{-1}.
\]

The {\it Evans function} may thus be defined via a ``shooting'' construction, 
similarly as in \cite{AGJ,GZ} for the whole line or \cite{R,SZ} for the half-line case, as
\be\label{evans}
D(\lambda):= \det ( \mathcal{Z}_1, \dots, \mathcal{Z}_{2n-r})|_{x=1},
\ee
where $\mathcal{Z}_1, \dots, \mathcal{Z}_{n-r}$ are solutions with data $\mathcal{Z}_j(0)=\bp 0,0, B_{22}e_j^{n-r}\ep$
prescribed at $x=0$ and 
$\mathcal{Z}_{n-r+1}, \dots, \mathcal{Z}_{n}$ 
and 
$\mathcal{Z}_{n+1}, \dots, \mathcal{Z}_{2n-r}$ 
are solutions with data 
$$
\mathcal{Z}_{n-r+j}(1)=\bp A_{11} e^r_j,0,0\ep,\quad j=1,\dots, r
$$
and
$ \mathcal{Z}_{n+j}(1)=\bp 0,0, B_{22}e_j^{n-r}\ep$, $j=1,\dots, n-r$
prescribed at $x=1$ and where $e^{p}_{j}$ is the j-th element of the standard basis of $\R^{p}$. Evidently, $D(\lambda)$ vanishes if and only if there exists a solution to the eigenvalue equation, hence {\it spectral stability is equivalent to nonvanishing of the Evans function on $\Re \lambda \geq 0$.}

\subsubsection{Stability index, uniqueness, and Zumbrun-Serre/Rousset lemma}
Clearly $D$ is real-valued for real $\lambda$.
It is readily seen (see, e.g., \cite{SZ,Z1} in the half-line case) 
that $D(\lambda)\neq 0$ for $\lambda$ real and sufficiently large, hence we may define as in \cite{GZ}
the {\it Stability index}
\begin{equation}\label{stabind}
\mu:= \sgn D(0)\left( \lim_{\lambda \to +\infty_{real}}\sgn D(\lambda) \right).
\end{equation}
The index $\mu$ determines the parity of the number of roots of the Evans function with positive real part, 
with $+1$ corresponding to ``even'' and $-1$ to ``odd''.
Thus, $\mu=+1$ is a {\it necessary condition for stability}.
The following result analogous to the Zumbrun-Serre/Rousset lemmas of \cite{ZS,R} in the whole- and half-line case,
relates low-frequency stability $D(0)\neq 0$ and stability index information to transversality of the steady profile
solution of the standing-wave ODE (cf. \cite[Lemma 6.1]{paper1} for gas dynamics).

\bl\label{ZSlem}
The zero-frequency limit $D(0)$ is equal to $\det B_{2,2}$ multiplied by the
Jacobian determinant $\det (d\Psi(C_2))$ associated with problem \eqref{ode_general} evaluated at any root $C_2$;
in particular,
\be\label{sgns}
\sgn D(0)= \sgn \det d\Psi(C_2).
\ee
\el

\begin{proof}
	This follows from the observation that for $\lambda=0$ the eigenvalue equation reduces to the 
	variational equation for \eqref{ode_general} for the steady profile, with $C_1=0$ and $C_2=B_{22}U_{II}'(0)$
	imposed by the homogeneous boundary conditions at $x=0$.
	Noting that $d\Psi(C_2)$ is equal to the value of the solution $U_{II}(1)$ of \eqref{varprof}
	with matrix-valued data $\tilde C_2=C_2=\Id_{n-r}$, whereas $D(0)$ is equal to the determinant
	of the solution with matrix-valued data $U_{II}'(0)=B_{22}\Id_{n-r}$, we have 
	$D(0)=\det B_{22} \det d\Psi(C_2)$, giving the result by $\det B_{22}>0$.
\end{proof}

\subsection{Symmetrizable systems}\label{s:symm}
The spectral condition \eqref{speccond} is satisfied for many physical systems, in particular
ones that are {\it symmetrizable} in the sense of \cite{KSh}. We first have the following technical lemma.

\begin{lemma}
If $\tilde{A}$ is symmetric and $\tilde{B}+\tilde{B}^{T}>0$, then $\sigma(\tilde{B}^{-1} \tilde{A}) \cap i \mathbb{R} \subset \{ 0 \}$.
\end{lemma}

\begin{proof}
\noindent If $\tilde{B}v = i \tau \tilde{A}v$ for $\tau \neq 0$, we get
\begin{equation*}
2i\tau \langle v,\tilde{A}v \rangle = \langle v, (\tilde{B}+\tilde{B}^{T})v\rangle + \langle v, (\tilde{B}-\tilde{B}^{T})v\rangle.
\end{equation*}
and since $\tilde{B}+\tilde{B}^{T}>0$, $v=0$.
\end{proof}

We recall that \eqref{conservation_law} is said to be symmetrizable  if:

\medskip
\medskip
\noindent (\textbf{H3}) There exists a smooth map $S:U \in \R^{n} \mapsto S(U)$ such that, for any $U \in \R^{n}$, 
$S(U)A^0(U)$ is symmetric positive definite, $S(U)= \begin{pmatrix} S_{11}(U) & 0 \\ 0 & S_{22}(U) \end{pmatrix}$, $S(U)A(U)$ is symmetric and $S_{22}(U)B_{22}(U) + (S_{22}(U)B_{22}(U))^{T}>0$.
\medskip
\medskip

Then we have the following useful lemma.

\begin{lemma}\label{symlem}
Under assumption (\textbf{H3}), $\sigma \left(B_{22}^{-1} \left(A_{22} - A_{21} A_{11}^{-1} A_{12} \right)\right) \cap i \mathbb{R} \subset \{ 0 \}$.
\end{lemma}

\begin{proof}
\noindent We note first that by assumption $S_{22}A_{21}=(S_{11}A_{12})^{T}$. Then we write
$$
B_{22}^{-1} \left(A_{22} - A_{21} A_{11}^{-1} A_{12} \right) = (S_{22} B_{22})^{-1} \left(S_{22} A_{22} - (S_{11}A_{12})^{T} (S_{11}A_{11})^{-1} S_{11}A_{12} \right) 
$$
and the result follows from the previous lemma.
\end{proof}

\bt\label{symmthm}
For symmetrizable systems under Conditions (\textbf{H0})-(\textbf{H3}), almost-constant solutions of almost 
constant data exist and are locally unique, nondegenerate, and spectrally stable.
\et

\begin{proof}
The existence, local uniqueness, and nondegeneracy are immediate consequences of Proposition 
	\ref{existence_almost_steady} and Lemma \ref{symlem}.
Furthermore, one can easily adapt \cite[Prop. 3.2]{MelZ} and prove that the spectrum of the linearized operator about a steady state $\hat{U}$ only contains eigenvalues. We then consider the eigenvalue problem
\begin{equation*}
	\lambda  A^0(\widehat {U})V + (A(\widehat{U}) V)_{x} = \left(B(\widehat{U}) V_{x} + dB(\widehat{U}) (V) \hat{U}_{x} \right)_{x} \text{ , } V(0) = 0\text{  ,  } V_{II}(1) = 0.
\end{equation*}
If $\hat{U} \equiv U_{0}$, one can check that
\begin{equation*}
	\begin{aligned}
		\Re(\lambda) \left(S(U_0)A^0(U_{0}) V,V \right)_{L^{2}(0,1)} &+  \left(S_{22}(U_{0}) B_{22}(U_{0})V_{IIx}, V_{IIx} \right)_{L^{2}(0,1)}  \\
		&\quad + \frac{1}{2} |\sqrt{S_{11}(U_{0}) A_{11}(U_{0})} V_{I}(1)|^{2} = 0.
	\end{aligned}
\end{equation*}
Noting, since $A_{11}(U)$ is symmetric for the inner product associated to $S_{11}(U)$, that Conditions (\textbf{H2})-(\textbf{H3}) give $S_{11}(U_{0})A_{11}(U_{0})>0$, we thus have that $U_{0}$ is spectrally stable. 
	
Then, for almost-constant steady states (meaning $\hat{U}_{x}$ small enough), we use an appropriate Goodman-type estimate. We note that there exists a constant $C>0$, such that for any $\alpha \geq 0 $ and weight $\varphi(x)=e^{-\alpha x}$, we have

\begin{align*}
	\Re(\lambda) &\left( \varphi  S(\hat{U}) A^0(\hat {U}) V,V \right)_{L^{2}} +\frac12 \varphi(1) \left| \sqrt{S_{11}(\hat{U}) A_{11}(\hat {U})} V_{I}(1) \right|^{2}\\
	&\leq -\frac{\alpha}{2} \left( \varphi S(\hat{U}) A(\hat{U}) V, V \right)_{L^{2}} + C  | \hat{U} |_{L^{\infty}} \left( \varphi V, V \right)_{L^{2}} + \frac{\alpha^2}{2} \left( \varphi S(\hat{U}) B(\hat{U}) V, V \right)_{L^{2}}  \\
	&- \left( \varphi S(\hat{U}) B(\hat{U}) V_{x}, V_{x} \right)_{L^{2}} + C  | \hat{U} |_{L^{\infty}} (1+ \alpha) |\sqrt{\varphi} V_{IIx} |_{L^{2}} | \sqrt{\varphi} V_{II} |_{L^{2}}.
\end{align*}
Then, since the system is symmetrizable, there exists some constants $a,b>0$ such that
\begin{equation*}
	\left[ -\frac{\alpha}{2} \varphi S(\hat{U}) A(\hat{U}) + C  | \hat{U} |_{L^{\infty}} \varphi I_{n} + \frac{\alpha^2}{2} \varphi S(\hat{U}) B(\hat{U}) \right]  \leq \left( -\frac{a \alpha }{2} + C  | \hat{U} |_{L^{\infty}} + \frac{b \alpha^2}{2} \right) \varphi I_{n}.
\end{equation*}
We choose $\alpha>0$ so that $-\frac{a \alpha}{2} +  C  | \hat{U} |_{L^{\infty}} + \frac{b \alpha^{2}}{2} = -\frac{a \alpha}{4}$ (which is always possible for $| \hat{U} |_{L^{\infty}}$ small enough). Thus we get
\begin{align*}
	\Re(\lambda) \left( \varphi  S(\hat{U}) A^0(\hat {U}) V,V \right)_{L^{2}} \leq &-\frac{a\alpha}{4} | \sqrt{\varphi} V |_{L^{2}}^2  - \left( \varphi S(\hat{U}) B(\hat{U}) V_{x}, V_{x} \right)_{L^{2}}\\
	 &+ C  | \hat{U} |_{L^{\infty}} (1+\alpha)  | \sqrt{\varphi} V_{IIx} |_{L^{2}} | \sqrt{\varphi} V_{II} |_{L^{2}}.
\end{align*}
Using Young's inequality together with the fact that $| \hat{U} |_{L^{\infty}}$ is small enough, we obtain
\begin{equation*}
	\Re(\lambda) \left( \varphi  S(\hat{U}) A^0(\hat {U}) V,V \right)_{L^{2}} \leq -\frac{a\alpha}{8} |\sqrt{\varphi} V |_{L^{2}}^{2}.
\end{equation*}
Thus, we may again conclude that $\hat{U}$ is spectrally stable. See  \cite[Section 3]{MelZ} for similar computations in the isentropic case.
%gas dynamic case.
Alternatively, one could conclude by continuous dependence on coefficients of solutions of the eigenvalue ODE,
as in \cite{Z1,GMWZ10}.
\end{proof}

\br
Note that, contrary to the whole line situation (see for instance \cite{Kaw_thesis,KSh,Z3}),
we do not assume Kawashima's genuine coupling condition \cite{Kaw_thesis} in Theorem \ref{symmthm}. The main reason behind 
this is that a steady state $\hat{U}$ of a purely hyperbolic system ($B\equiv 0$) on an interval under assumptions (\textbf{H0}),(\textbf{H2}),(\textbf{H3}) is stable, 
as in this case $f=f_I$ and $f_I(U_0)=f_I(\hat U)$ enforces $\hat U$
constant, and also all characteristics move from left to right, sweeping perturbations out of the domain,
with full Dirichlet conditions at the left boundary.
Even better, any solution of Problem \eqref{conservation_law}-\eqref{BC2} initially close  enough to $\hat{U}$ is equal to $\hat{U}$ after a finite time.
Thus, the hyperbolic part of the system need not be coupled to the parabolic part in order to achieve
stability, but is stable even by itself.
\er

\begin{remark}\label{nonlinear_stab_remark}
Using the same kind of energy estimates, one can also prove the \emph{nonlinear stability} of 
an almost-constant steady state. See \cite[Section 6]{MelZ} for similar considerations 
in the isentropic gas dynamic case.
\end{remark}

Combining the argument of \ref{steadyprop} with the results of Theorem \ref{symmthm},
we may deduce not only nonvanishing of $D(0)/\mu$, $d\Psi$ 
but useful sign information, included here for definiteness.

\begin{proposition} \label{dPsi_eval_rmk}
For arbitrary solutions of symmmetrizable systems, 
\be\label{symm_mu}
\mu=\sgn D(0)=\sgn \det d\Psi(C_2).
\ee
For constant solutions of symmmetrizable systems, corresponding to $C_2=C_2^*$,
\be\label{symmcomp}
\mu=\sgn \det d\Psi(C^*_2)=+1.
\ee
\end{proposition}

\begin{proof}
From the argument of Proposition \ref{steadyprop}, we find that the profile map $\Psi:C_2\to U_{II}(1)$
satisfies $\det d\Psi(C_2^*)=\det \int_0^1 e^{\tilde A(1-s)}ds\neq 0$ at the value $C_2^*$ corresponding to a constant
solution, where $\tilde{A}=B_{22}^{-1} (A_{22} - A_{21} A_{11}^{-1} A_{12})$, and thus is invariant under
homotopy within the class of symmetrizable systems, hence may take $B_{22}$ without loss of generality to
be a multiple of the identity, hence, by symmetrizability, $\tilde A$ has real semisimple eigenvalues
$\tilde \alpha_1, \dots, \tilde \alpha_{n-r}$.
Diagonalizing, we have by direct computation that 
$\sgn \det d\Psi(C_2^*)=\sgn \det \int_0^1 e^{\tilde A(1-s)}ds= \Pi_{j=1}^{n-r}\sgn \int_0^1 e^{\alpha_j(1-s)}ds=1$,
hence $\sgn \det d\Psi(C^*_2)=\sgn D(0)=+1.$

Using spectral stability, $D(\lambda)\neq0$ together with real-valuedness of the Evans function when restricted to
the real axis, we may extend by a further homotopy in $\lambda$ to conclude that the limit of $\sgn D(\lambda)$
	as $\lambda \to \infty$, real, is also $+1$.  Summarizing, we have \eqref{symmcomp}.
By standard arguments \cite{SZ,Z1}, either WKB-type or using energy estimates, one has for general
solutions of symmetrizable systems that $D(\lambda)\neq 0$ for $\lambda$ real and large,
	hence $\lim_{\lambda \to +\infty} \sgn D(\lambda) =+1$, giving \eqref{symm_mu}, by homotopy in $C_2$. 
\end{proof}

\subsection{Systems with convex entropy}\label{s:entropy}
A system \eqref{conservation_law} is said to have a convex entropy \cite{Kaw_thesis,KSh} if it has
an entropy/entropy flux pair $(\eta, q)(f^0):\R^n\to\R^{2}$ such that 
\begin{equation}\label{ent1}
	\frac{d^2\eta}{(df^0)^2} >0 \text{  ,  } (d\eta/df^0) (df/df^0)= dq/df^0, 
	%(d/df^0)^2 \eta >0, \quad (d\eta /df^0)(df/df^0)= (dq/df^0), 
\end{equation}
and
\begin{equation}\label{ent2}
	(df^0)^{T} \frac{d^2 \eta}{(df^0)^2}  B+ \Big( (df^0)^{T} \frac{d^2 \eta}{(df^0)^{2}}  B\Big)^t \geq 0,
	%(df^0)^{T} (d/df^0)^2 \eta  B+ ( (df^0)^{T} (d /df^0)^{2} \eta  B)^t \geq 0,
\end{equation}
with equality only on  $\ker B$.
It is a theorem of \cite{KSh} that existence of a convex entropy implies symmetrizability, i.e.,
reducibility by coordinate change to a system satisfying (\textbf{H3}). Thus, we may deduce local 
uniqueness information for systems with a convex entropy already
by reference to Theorem \ref{symmthm}.

\subsubsection{Global uniqueness for constant data}
Arguing directly, we may obtain under a mild additional assumption, much more.
Namely, assume as holds for most physical systems that
\be\label{fsymm}
\hbox{\rm $(df_I)_I$ symmetric (hence positive definite).}
\ee
For gas dynamics, magnetohydrodynamics, and elasticity, $(df_I)_I = u\Id_r$ and so \eqref{fsymm} is
trivially satisfied.
By $\langle h, f_I(W+h, U_{II})-f_I(W)\rangle=\langle h,  \big(\int (df_I)_I(W+\theta h)d\theta \big) h\rangle>0$, 
this yields the global solvability property
\be\label{solve}
\hbox{\rm For fixed $U_{II}$, $(df_I)(\cdot, U_{II})$ is (globally) one-to-one.}
\ee

Alternatively, we may take $(df_I)_I+ (df_I)_I^t>0$,
or just impose \eqref{solve} without reference to \eqref{fsymm}.
Then, we have the following \emph{global} uniqueness result, for constant data only.

\bt\label{entthm}
For systems \eqref{conservation_law}-\eqref{BC2} with a global convex entropy and 
satisfying (\textbf{H0})-(\textbf{H3}) and \eqref{solve}, 
solutions of \eqref{steady_conservation_law} for constant data $U_{0II}=U_{1II}$
are \emph{globally unique}, nondegenerate (full rank), and spectrally stable, 
consisting exclusively of constant states.
\et

\begin{proof}
	Following \cite{La}, we obtain multiplying \eqref{conservation_law} by $d\eta/df^0$
	and using \eqref{ent1}(ii) the equation
$$
	\eta_t + q_x= (d\eta/ df^0) (B U_x)_x = ((d\eta/ df^0) BU_x)_x 
	- \langle df_0 U_x, \frac{d^2 \eta}{ (df^0)^2} U_x\rangle .
	%- \langle (df_0 U_x, (d/df^0)^2\eta  B U_x\rangle .
	$$
	By \eqref{ent2}, we have 
	%$ \langle df_0 U_x, (d/df^0)^2\eta B U_x\rangle \geq 0$ with equality if and only if $BU_x=0$.
	$ \langle df_0 U_x, \frac{ d^2\eta}{(df^0)^2} B U_x\rangle \geq 0$ with equality if and only if $BU_x=0$.
Thus, integrating the steady equation from $x=0$ to $x=1$, we obtain
\be\label{keyeq}
(q(U) - d\eta(U) B(U)U')|_0^1\leq 0,
\ee
with equality if and only if 
	%$ \langle U', d^2\eta B U'\rangle \equiv 0$, 
	$BU'\equiv 0$ or equivalently $U_{II}'\equiv 0$.

On the other hand,
	integrating the $U_I$ equation, we have $f_I(U)\equiv \const$, whence, by (\textbf{H2}) and \eqref{solve},
$U_I(0)=U_I(1)$, and so $U(0)=U(1)$.  Thus, $q(U)|_0^1$ vanishes in \eqref{keyeq}.
At the same time, by addition of an arbitrary linear function, we may take $\eta$ without loss of generality
to satisfy $d\eta(U(0))=d\eta(U(1))=0$, whence the entire lefthand side of \eqref{keyeq} vanishes.
	Then, equality holds in \eqref{keyeq} and so we must have $U_{II}'\equiv 0$ and therefore
	$U_{II}\equiv  U_{II}(0)$ and $f_I\equiv f_I(U(0))$.
Applying (\textbf{H2}) and the implicit function theorem, we find that we may solve for a unique value
	of $U_I$ in a neighborhood of $U_I(0)$ as a function of the constants $U_{II}$ and $f_I(U)$.
	Since $U_I(x)$ varies continuously starting at $U_I(0)$, it can thus never escape this local neighborhood,
	and so $U_I(x)\equiv U_I(0)$ as well.
	This gives global uniqueness of $U\equiv U(0)$.
Nondegeneracy and spectral stability follow by Theorem \ref{symmthm}.
\end{proof}

\br\label{symmrmk}
The assumption \eqref{fsymm} seems possibly related to the circle of ideas around entropy and symmetrizability.
It would be very interesting to identify sharp criteria for \eqref{fsymm} assuming existence of
a complex entropy; however, we have not found such.
\er

\subsubsection{Global existence for general data}
Following \cite{paper1}, define the feasible set $\mathcal{C}$ as the connected component of the set of parameters $C_2$
corresponding to constant solutions of the open set of $C_2$ for which the solution of \eqref{ode_general}
is defined on $[0,1]$ and remains in the interior of its domain of definition $\mathcal{U}$.
If $\Psi$ is ``proper'' in the sense that the inverse image in $\mathcal{U}$ of a compact set
in $\mathcal{C}$ is compact, then we may conclude general existence from the special uniqueness result of
Theorem \ref{entthm}.

\begin{corollary}\label{gexist}
	For systems possessing a convex entropy for which $\Psi$ is proper in the above sense with repect to 
	$\mathcal{U}$, $\mathcal{C}$, there exists a steady profile solution for any data $U(0), U_2(1)$.
\end{corollary}

\begin{proof}
	The map $\Psi$ is continuous on $\mathcal{C}$, by continuous dependence of solutions of ODE,
	For $\Psi$ proper in the above sense, one may thus define the Brouwer degree 
	$d(\Psi, \mathcal{C}, U_{II}(1))$ for each target image $U_{II}(1)$, considering left data $U(0)$
	as a fixed parameter \cite{Brouwer1911,Mil1965,Hi1976,Pra2006,DincaMawhin2021}.
	Recall that Brouwer degree of a proper map is invariant under homotopy, and, for regular
	values $U_{II}(1)$, for which $\Psi^{-1}$ consists of a finite set of isolated nondegenerate roots,
	is given by
\be\label{regdegree}
	d(\Psi,\mathcal{C},U_{II}(1)):=\sum_{C_2 \in \Psi^{-1}(U_{II}(1))} \sgn \det d\Psi(C_2). 
\ee
	This includes the case $\Psi^{-1}(U_{II}(1))=\emptyset$, for which
	$d(\Psi,\mathcal{C},U_{II}(1))=0$, hence nonzero Brouwer degree implies existence of a solution.

	By homotopy invariance, the degree may be computed at any such $U_{II}(1)$, in particular the
	value $U_{II}^*(1)=U_{II}(0)$ to which Theorem \ref{entthm} applies. For this value, the inverse image
	of $\Psi$ consists of a single $C_2^*$ corresponding to a constant solution.
	Moreover, by Proposition \ref{dPsi_eval_rmk}, $C_2^*$ is an isolated nondegenerate root,
	with $\sgn \det d\Psi(C_2^*)=+1$, hence by \eqref{regdegree}
	$d(\Psi,\mathcal{C},U_{II}^*(1))=+1$.  By homotopy invariance, therefore,
	$d(\Psi,\mathcal{C},U_{II}(1))=+1$ for any $U_{II}(1)\in \mathcal{U}^{int}$.
	implying existence of a solution by nonvanishing of the degree.
\end{proof}

\br\label{proper_rmk}
The condition that $\Psi$ be proper is quite strict in practice, as solutions may escape to infinity, etc.
It was verified in \cite{paper1} for the compressible Navier--Stokes equations with polytropic
equation of state.  It is a very interesting open question whether it holds for a general convex equation of state,
or for more complicated physical systems such as viscoelasticity and  magnetohydrodynamics (MHD).
\er

\br\label{elliprmk}
It is worth noting that under solvability assumption \eqref{solve}, 
the shock tube problem (namely the Navier-Stokes equations in a tube) reduces to a nonlinear elliptic problem in $U_{II}$, with 
Dirichlet boundary conditions, to which all of the theory of global existence and uniqueness
for such may be applied.
However, up to now we have not been able to make use of this connection; rather our studies of specific
systems seem to be an interesting source of examples for elliptic theory.
\er

\subsubsection{Uniqueness and spectral stability}
For large-amplitude solutions, neither uniqueness nor spectral stability appear to follow from the existence
of a convex entropy.  Indeed, numerical computations of \cite{paper1} indicate that for the common example of
the compressible Navier--Stokes equation with an artificially chosen convex equation of state- hence
possessing a convex entropy- both failure of local uniqueness, and failure of spectral stability may occur,
with steady and Hopf bifurcations.

\subsection{Conclusions}
For small-amplitude data, local existence, uniqueness, and spectral stability may be deduced 
from the standard structural assumptions of symmetrizability. 
For systems with convex entropy, we obtain, further, from global uniqueness of solutions of constant data.
{Provided the mapping $\Psi$ may to shown to be proper} (in the sense described above), we obtain as a corollary
large-amplitude existence; however, this appears possibly quite restrictive.
The (numerical) gas-dynamical examples of \cite{paper1} suggest that
large-amplitude uniqueness and spectral stability {\it in general do not hold} for systems possessing a convex 
entropy.
Thus, these questions must apparently be addressed either by problem-specific analysis, asymptotic limit, or
numerical Evans function evaluation as in \cite{MelZ,paper1}

\section{Asymptotic limits I : the inviscid limit in the scalar case}\label{s:asymptoticI}

In this section we study the behavior of steady solutions of scalar viscous conservation laws. More precisely, we consider the problem

\begin{equation}\label{steady_scalar}
\left\{
\begin{array}{l}
\eps \hat u^{\eps}_{xx} =(f(\hat u^{\eps}))_{x} \text{  ,   } 0<x<1,\\
\hat u^{\eps}(0)=u_{0} \text{  ,  } \hat u^{\eps}(1)=u_{1},
\end{array}
\right.
\end{equation}
for $u_{0},u_{1} \in \R$ and $\eps>0$. We will assume in this section that 
\begin{enumerate}
  \item[(A0)] $f$ is a $\mathcal{C}^{2}$ strictly convex function satisfying $f(0)=f'(0)=0$.
\end{enumerate}
We will also assume sometimes that
\begin{enumerate}
  \item[(A1)] $f''(0) \neq 0$.
\end{enumerate}

The second assumption is not necessary but we will suppose it sometimes by simplicity when we will deal with characteristic boundary layers and double boundary layers. Our goal is to give a description of these steady solutions especially when $\eps$ goes to $0$. First, we give an existence result.

\bpr\label{existence_steady}
Assume that $f$ is $\mathcal{C}^{1}(\R)$. For any $\eps>0$ and any $u_{0},u_{1} \in \R$, there exists a unique solution of \eqref{steady_scalar}. 
\epr

In the following, we will denote by $\hat u^{\eps}$ the unique solution of \eqref{steady_scalar}. 

\begin{proof}
The arguments used are based on the proof of Proposition 2.1 in \cite{MelZ}. We introduce the ODE with parameter $b \in \R$
\[
\eps y_{b}' =b + f(y_{b}) \text{  ,  } y_{b}(0)=u_{0}
\]
and denote by $y_{b}$ the unique maximal solution. Note that $y_{b}$ is monotonic. We define the set $I_{u_{0}}$ as follows: $b \in I_{u_{0}}$ if and only if $y_{b}$ is defined on $[0,1]$. The set $I_{u_{0}}$ is open and $-f(u_{0}) \in I_{u_{0}}$. Furthermore, by the comparison theorem for ODE, if $b_{1}<b_{2}$, then $y_{b_{1}}<y_{b_{2}}$ on the intersection of the domains of definition of the two functions. Therefore $I_{u_{0}}$ is an interval. Then, we introduce the map
\[
\phi : b \in I_{u_{0}} \mapsto y_{b}(1)-u_{1}.
\]
The comparison theorem for ODE asserts that $\phi$ is increasing. Furthermore, $\phi$ is not bounded from below and not bounded from above.  Therefore, there exists a unique $b \in I_{u_{0}}$ such that $\phi(b)=u_{1}$.
\end{proof}

The purpose of this section is to describe the behavior of $\hat u^{\eps}$. If $u_{0}=u_{1}$ then $\hat u^{\eps}$ is constant. If $u_{0}>u_{1}$, $\hat u^{\eps}$ is decreasing and we talk about compressive waves. If $u_{0}<u_{1}$, $\hat u^{\eps}$ is increasing and we talk about expansive waves.

Before that, we would like to give a motivation of studying such solutions. We will show that there are spectrally stable (and then nonlinearly stable by adapting Section 6 of \cite{MelZ}). We introduce the unbounded operator $\mathcal{L}$ on $L^{2}(0,1)$ with domain $\mathcal{D}(\mathcal{L}) = H^{2}(0,1) \cap H^{1}_{0}(0,1)$ as 
\[
\mathcal{L} v =  \eps v'' - (f'(\hat u^{\eps}) v)'.
\] 
We denote by $\sigma(\mathcal{L})$ the spectrum of $(\mathcal{L},\mathcal{D}(\mathcal{L}))$.

\bl\label{stab_steady}
Assume that $f$ is $\mathcal{C}^{2}(\R)$. Let $\eps>0$ and $u_{0},u_{1} \in \R$. The inverse of $\mathcal{L}$ exists and is compact. The set $\sigma(\mathcal{L})$ only contains eigenvalues.
\el

\begin{proof}
The arguments used are based on the proof of Proposition 3.2 in \cite{MelZ}. For any $g \in L^{2}(0,1)$ the system 
\[
\left\{
\begin{array}{l}
\eps \hat v'' =(f'(\hat u^{\eps}) v)' + g\\
v(0)=0 \text{  ,  } v(1)=0
\end{array}
\right.
\]
admits the unique solution
\[
v(x) = \int_{0}^{x} \exp \left( - \int_{y}^{x} f'(\hat u^{\eps}) \frac{dz}{\eps}\right) \left[ \int_{0}^{y} g(z) \frac{dz}{\eps} - \frac{\exp \left( - \int_{y}^{1} f'(\hat u^{\eps}) \frac{dy}{\eps}\right) \int_{0}^{y} g(z) \frac{dz}{\eps}  }{\exp \left( - \int_{y}^{1} f'(\hat u^{\eps}) \frac{dy}{\eps}\right)} \right] dy
\]
so that $\mathcal{L}$ is invertible. Furthermore, if $g$ is in a ball of $L^{2}(0,1)$, the previous formula show that $v$ is in a ball of $H^{1}_{0}(0,1)$ so that, by Rellich--Kondrachov theorem, $\mathcal{L}^{-1}$ is compact.
\end{proof}
We can now state the spectral stability of $\hat u^{\eps}$. The following proposition corresponds to Lemma 2.5 in \cite{kreiss_burger_interval}. 
\bpr
Assume that $f$ is $\mathcal{C}^{2}(\R)$.  Any eigenvalue of $\mathcal{L}$ is real and negative. 
\epr

\subsection{Compressive waves}
This subsection is devoted to the study of compressive waves, i.e. $u_{0}>u_{1}$. The main result of this subsection asserts that $\hat u^{\eps}$ is a piece of a whole-line steady shock wave. In the following, we define, for $y \neq 0$, the quantity $y^{\dag}$ as the unique solution of the equation $f(x)=f(y)$ with $x \neq y$. We first recall some facts about whole-line shocks. Let $a>0$. We introduce the steady shock ODE
\begin{equation}\label{steady_shock}
\left\{
\begin{array}{l}
u'' =(f(u))' \text{  ,   } x \in \R,\\
u(-\infty)=a \text{  ,  } u(+\infty)=a^{\dag}.
\end{array}
\right.
\end{equation}
This problem is invariant under translation so that it is convenient to impose the extra condition
\begin{equation}\label{fix_shock}
u(0)=0.
\end{equation}
We recall some basic properties of the solution of this problem.
\bl\label{l.whole.line.shock}
Assume that $f$ satisfies (A0). For any $a>0$, Problem \eqref{steady_shock}-\eqref{fix_shock} admits a unique solution $u$. The function $u$ is decreasing. Furthermore, there exists a constant $C_{1}>0$ such that for any $x \leq 0$,
\[
 |u(x)-a| + |u'(x)| \leq C_{1} e^{f'(a) x},
 \]
and there exists a constant $C_{2}>0$ such that for any $x \geq 0$,
\[
|u(x)-a^{\dag}| +  |u'(x)| \leq C_{2} e^{f'(a^{\dag}) x}.
\]
%Finally, for any $x<0$, $\displaystyle \lim_{a \to +\infty} u(x) = +\infty$ and for any $x>0$, $\displaystyle \lim_{a \to +\infty} u(x) = -\infty$.
\el
In the following we denote by $u(\cdot;a)$ this unique solution. Note that $u \left( \frac{\cdot}{\eps};a \right)$ is the unique solution of
\[
\left\{
\begin{array}{l}
\eps u'' =(f(u))' \text{  ,   } x \in \R,\\
u(-\infty)=a \text{  ,  } u(+\infty)=a^{\dag}.
\end{array}
\right.
\]
\begin{proof}
We briefly recall the proof of this statement. It is convenient to consider the equivalent problem
\begin{equation}\label{steady_shock2}
\left\{
\begin{array}{l}
u' =f(u)-f(a) \text{  ,   } x \in \R,\\
u(0)=0.
\end{array}
\right.
\end{equation}
Introducing the map 
\[
\phi_{a} : y \in (a^{\dag},a) \mapsto \int_{0}^{y} \frac{dz}{f(z)-f(a)} \in \R,
\]
we note that $\phi_{a}$ is decreasing and bijective. we define $u(x) := \phi_{a}^{-1}(x)$. Then, we notice that
\[
(u')' = f'(u) u' \text{  ,  } u'(0)= -f(a),
\]
so that $u'(x) = - f(a)\exp \left( \int_{0}^{x} f'(u(y)) dy \right)$. We get that $u'$ and $u-a$ go exponentially to $0$ as $x \to -\infty$. Therefore, the function $x \in \R^{-} \mapsto \int_{0}^{x} f'(u(y)) dy - f'(a) x$ is bounded. The estimates on $\R^{-}$ follow easily. The estimates on $\R^{+}$ are similar. 
\end{proof}
The next lemma provides an estimate of the differential of $u(\cdot;a)$ with respect to $a$.
\bl\label{l.diff_a}
Assume that $f$ satisfies (A0). Let $a>0$ and $b \in \R$. Denote $u=u(\cdot;a)$ and $v_{b} = \displaystyle \lim_{s \to 0} \frac{u(\cdot;a+sb)-u(\cdot;a)}{s}$. Then, there exists a constant $C>0$ depending only on $a$ such that, for any $x \in \R$,
\[
|v_{b} (x)| \leq C |b|.
\]
%Furthermore, for any $x<0$, $\psi_{x}$ is increasing and bijective and, for any $x>0$, $\psi_{x}$ is decreasing and bijective.
\el
\begin{proof}
The function $v$ satisfies the ODE
\[
v_{b}' =f'(u)v_{b} -f'(a)b \text{   ,   } v(0)=0,
\]
so that $v_{b}(x)= - b f'(a) \int_{0}^{x} \exp \left( \int_{y}^{x} f'(u) dz \right) dy$. The estimate follows from the estimates of Lemma \ref{l.whole.line.shock}. 
\end{proof}
We get from Lemma \ref{l.whole.line.shock} that for any $a>0$, for any $y \in (a^{\dag},a)$, the equation $u(x;a)=y$ has a unique solution. We keep the previous notation and denote as $\phi_{a}(y)$ this unique solution. 
%autre nom $v(y;a)$
We can now state the main result of this section.
\bt 
Assume that $f$ satisfies (A0). Let $\eps>0$, $u_{0},u_{1} \in \R$ such that $u_{0}>u_{1}$. Then the unique solution $\hat u^{\eps}$ of Problem \eqref{steady_scalar} is the restriction to $[0,1]$ of a whole-line shock: there exists a unique pair $(a,\lambda) \in \R^{+}_{\ast} \times \R$ such that, for any $x \in [0,1]$, we have $\hat u^{\eps}(x)=u \left( \frac{x}{\eps}+\lambda ; a \right)$. Furthermore,
\begin{enumerate}[label=\arabic*)]
\item If $f(u_{0})>f(u_{1})$, the shock moves to the endpoint $x=1$ as $\eps \to 0$: $u_{0}>0$ and there exists a constant $C>0$ such that, for any $x \in [0,1]$,
\[
\left| \hat u^{\eps}(x) - u \left( \frac{x-1}{\eps}+\phi_{u_{0}}(u_{1}) ; u_{0} \right) \right| \leq C \exp \left( - \frac{f'(u_{0})}{\eps} \right).
\]
\item If $f(u_{0})<f(u_{1})$, the shock moves to the endpoint $x=0$ as $\eps \to 0$: $u_{1}<0$ and there exists a constant $C>0$ such that, for any $x \in [0,1]$,
\[
\left| \hat u^{\eps}(x) - u \left(\frac{x}{\eps}+\phi_{u_{1}^{\dag}}(u_{0}) ; u_{1}^{\dag} \right) \right| \leq C \exp \left(\frac{f'(u_{1})}{\eps} \right).
\]
\item If $f(u_{0})=f(u_{1})$, the shock stays inside the interval as $\eps \to 0$: $u_{0}>0>u_{1}=u_{0}^{\dag}$ and there exists a constant $C>0$ such that, for any $x \in [0,1]$,
\[
\left| \hat u^{\eps}(x) - u \left(\frac{x-\frac{f'(u_{0}^{\dag})}{f'(u_{0}^{\dag})-f'(u_{0})}}{\eps} +b ; u_{0} \right) \right| \leq \frac{C}{\eps} \exp \left( - \frac{f'(u_{0}^{\dag}) f'(u_{0})}{f'(u_{0}^{\dag})-f'(u_{0})} \frac{1}{\eps} \right).
\]
with
\[
b=\frac{1}{f'(u_{0}^{\dag})-f'(u_{0})} \left( \int_{u_{0}^{\dag}}^{0} \frac{f'(u_{0}^{\dag})-f'(z)}{f(u_{0})-f(z)} dz + \int_{0}^{u_{0}} \frac{f'(u_{0})-f'(z)}{f(u_{0})-f(z)} dz\right).
\]
\end{enumerate}
\et

\br 
Note that the previous theorem also works with characteristic endpoints, i.e. the cases $u_{0}=0 > u_{1}$ or $u_{0} > u_{1}=0$. Therefore, in these cases, one expect a boundary layer of size $\eps$ whereas characteristic boundary layers on the half-life have size $\sqrt{\eps}$.
\er

\br 
When $f$ is even, we have $u_{0}^{\dag}=-u_{0}$ and in the case $u_{1}=-u_{0}$ a better look of the proof shows that
\[
\left| \hat u^{\eps}(x) - u \left(\frac{x-\frac12}{\eps} ; u_{0} \right) \right| \leq C \exp \left( - \frac{f'(u_{0})}{2\eps} \right).
\]
\er 

\begin{proof}
We consider first the case $f(u_{0})>f(u_{1})$. In that case $u_{0}>0$ and $u_{1} \in (u_{0}^{\dag},u_{0})$. We solve the following system, with unknowns $(a,\lambda) \in (u_{0},+\infty) \times \R^{-}$,
\[
\left\{
\begin{array}{l}
u(\lambda;a)=u_{0}\\
u \left( \frac{1}{\eps} + \lambda; a\right)= u_{1}
\end{array}
\right.
\Longleftrightarrow
\left\{
\begin{array}{l}
\lambda=\phi_{a}(u_{0})\\
\frac{1}{\eps} + \phi_{a}(u_{0}) = \phi_{a}(u_{1}).
\end{array}
\right.
\]
Then, we remark that $\psi := a \in (u_{0},+\infty) \mapsto \phi_{a}(u_{1}) - \phi_{a}(u_{0}) \in \R^{+}_{\ast}$ is actually
\[
\psi(a) = \int_{u_{1}}^{u_{0}} \frac{dz}{f(a)-f(z)}
\]
and this map  is continuous and decreasing, $\displaystyle \lim_{a \to u_{0}+} \psi(a)=+\infty$ and $\displaystyle \lim_{a \to +\infty} \psi(a)=0$. That proves that $\hat u^{\eps}$ is the restriction to $[0,1]$ of a unique whole-line shock. Furthermore, we remark  that $\displaystyle \lim_{\eps \to 0} a=u_{0}$. Then, denoting by simplicity $f^{-1}$ as the inverse function of $f_{|\R^{+}}$, we have
\begin{align*}
\int_{0}^{u_{0}} \hspace{-0.1cm} \frac{dz}{f(a)-f(z)} - \frac{1}{f'(u_{0})} \hspace{-0.1cm}  \int_{0}^{f(u_{0})} \frac{dy}{f(a)-y} &= \int_{0}^{f(u_{0})} \hspace{-0.1cm}  \frac{1}{f'(f^{-1}(y)) f'(u_{0})} \frac{f'(u_{0}) - f'(f^{-1}(y))}{f(a)-y} dy\\
&= \frac{1}{f'(u_{0})}  \int_{0}^{u_{0}} \frac{f'(u_{0})-f'(z)}{f(a)-f(z)} dz
\end{align*}
the last expression being bounded uniformly with respect to $\eps \to 0$. Therefore, there exists a constant $C>0$ such that
\[
0 \leq f'(u_{0})(a-u_{0}) \leq f(a)-f(u_{0}) \leq C e^{-\frac{f'(u_{0})}{\eps}}.
\]
That proves that $a = u_{0} + \mathcal{O} ( e^{-\frac{f'(u_{0})}{\eps}} )$. Furthermore, we have
\[
\lambda = - \frac{1}{\eps} + \phi_{a}(u_{1}) = - \frac{1}{\eps} + \phi_{u_{0}}(u_{1}) + \mathcal{O} ( e^{-\frac{f'(u_{0})}{\eps}} ).
\]
Finally, for any $x \in [0,1]$, using Lemma \ref{l.diff_a} and the fact $u'(\cdot;u_{0})$ is bounded, there exists a constant $C>0$ such that
\begin{align*}
\left| u \left( \frac{x}{\eps}+\lambda ; a \right)  - u \left( \frac{x-1}{\eps}+\phi_{u_{0}}(u_{1})  ; a\right) \right| &\leq \left| u \left( \frac{x}{\eps}+\lambda ; a \right)  - u \left( \frac{x}{\eps}+\lambda ; u_{0} \right) \right|\\
&\hspace{0.5cm} +\left| u \left( \frac{x}{\eps}+\lambda ; u_{0} \right)  - u \left( \frac{x-1}{\eps}+\phi_{u_{0}}(u_{1}) ; u_{0} \right) \right|\\
&\leq C  \left| a - u_{0} \right| +C \left| \lambda + \frac{1}{\eps} +\phi_{u_{0}}(u_{1}) \right|.
\end{align*}
The case $f(u_{0})<f(u_{1})$ follows similarly. If now $f(u_{0})=f(u_{1})$, following the previous argument, there exists a constant $C>0$ independent of $\eps \to 0$ such that
\[
\left| \int_{u_{0}^{\dag}}^{u_{0}} \hspace{-0.1cm} \frac{dz}{f(a)-f(z)} - \frac{1}{f'(u_{0})} \hspace{-0.1cm}  \int_{0}^{f(u_{0})} \frac{dy}{f(a)-y}  - \frac{1}{f'(u_{0}^{\dag})} \hspace{-0.1cm}  \int_{f(u_{0})}^{0} \frac{dy}{f(a)-y} \right| \leq C
\]
so that $a = u_{0} + \mathcal{O} \left( e^{- \frac{f'(u_{0}^{\dag}) f'(u_{0})}{f'(u_{0}^{\dag})-f'(u_{0})} \frac{1}{\eps}} \right)$. Then, we get after some computations that
\begin{align*}
\lambda = &- \frac{f'(u_{0}^{\dag})}{f'(u_{0}^{\dag})-f'(u_{0})} \frac{1}{\eps} + \frac{1}{f'(u_{0}^{\dag})-f'(u_{0})} \left( \int_{u_{0}^{\dag}}^{0} \frac{f'(u_{0}^{\dag})-f'(z)}{f(u_{0})-f(z)} dz + \int_{0}^{u_{0}} \frac{f'(u_{0})-f'(z)}{f(u_{0})-f(z)} dz\right)\\
&+ \frac{1}{f'(u_{0}^{\dag})-f'(u_{0})} c_{a}
\end{align*}
with
\[
c _{a}= \int_{u_{0}^{\dag}}^{0} \frac{f'(u_{0}^{\dag})-f'(z)}{f(u_{0}^{\dag})-f(z)} \frac{f(u_{0})-f(a)}{f(a)-f(z)} dz + \int_{0}^{u_{0}} \frac{f'(u_{0})-f'(z)}{f(u_{0})-f(z)} \frac{f(u_{0})-f(a)}{f(a)-f(z)} dz.
\]
It is easy to check that there exists a constant $C>0$ such that
\[
|c_{a}| \leq C (a-u_{0}) (1+|\ln(a-u_{0})|)
\]
and the result follows from the same way as the previous cases.
\end{proof}

\subsection{Expansive waves}
This subsection is devoted to the study of expansive waves, i.e. $u_{0}<u_{1}$. We will consider two different cases. Firstly, we will study the case where $u_{0}$ and $u_{1}$ agree in sign. In that case, we obtain, as $\eps \to 0$, a single boundary layer. Secondly, we will consider the situation where $u_{0}<0<u_{1}$. We will see in that case that a double boundary layer occurs as $\eps \to 0$.

\subsubsection{Single boundary layer at $x=1$}\label{singleBL1}
We start with the case $0 \leq u_{0}<u_{1}$. The flow is outgoing at $x=1$ and we expect a boundary layer at $x=1$. Influence by the work of Gisclon-Serre {\cite{GS94,G96} on the half-line, we consider the following expansion, for any $x \in [0,1]$,
\begin{equation}\label{WKB_1}
\hat u^{\eps}(x) = u_{0} + v \left( \frac{1-x}{\eps} \right) + w \left( \frac{1-x}{\eps} \right)
\end{equation}
where $v$ is a function defined on $\R^{+}$ that satisfies the following ODE
\begin{equation}\label{ODE_BL1}
v'' = -(f(u_{0}+v))' \text{   ,   } v(0)=u_{1}-u_{0} \text{  and  } v(+\infty)=0,
\end{equation}
and $w$ a function defined on $\left[0,\frac{1}{\eps}\right]$ supposed to be small and satisfying $w(0)=0$. We begin our analysis with a technical lemma about the solution of Problem \eqref{ODE_BL1}.

\bl\label{v_BL1}
Assume that $f$ satisfies (A0). Let $u_{0},u_{1} \in \R$ such that $0 \leq u_{0}<u_{1}$. Assume that $v$ solves \eqref{ODE_BL1}. Then $v$ is positive, $v'$ is negative. There exists a constant $C>0$ such that, for any $y \in \R^{+}$, if $u_{0}>0$ we have
\[
|v(y)| \leq C e^{-f'(u_{0})y},
\]
whereas if $u_{0}=0$ and $f$ satisfies (A1) we have
\[
|v(y)| \leq \frac{C}{1+y}.
\]
\el 

\begin{proof}
The proof is similar to the one of Lemma \ref{l.whole.line.shock}. We consider the equivalent problem by integrating from $+\infty$ to $y \geq 0$
\[
v' =f(u_{0})-f(u_{0}+v) \text{   ,   } v(0)=u_{1}-u_{0}.
\]
Introducing the map 
\[
\phi : y \in (0,u_{1}-u_{0}] \mapsto \int_{y}^{u_{1}-u_{0}} \frac{dz}{f(u_{0}+y)-f(u_{0})} \in \R_{+},
\]
we note that $\phi$ is decreasing and bijective. Then, $v(x) = \phi^{-1}(x)$. If $u_{0}>0$, the estimate follows with the same argument as in the proof of Lemma \ref{l.whole.line.shock} whereas if $u_{0}=0$ we use the fact that under assumption (A1) there exists a constant $\alpha>0$ such that $f(y) \geq \alpha y^2$ for $y>0$ small enough.
\end{proof}

We can now state the main result of this subsubsection.
\bt\label{thm_BL1}
Assume that $f$ satisfies (A0). Let $\eps \in (0,1]$, $u_{0},u_{1} \in \R$ such that $0 \leq u_{0}<u_{1}$. Let $\hat u^{\eps}$ be the unique solution of Problem \eqref{steady_scalar} and $v$ the unique solution of Problem \eqref{ODE_BL1}. Then, there exists a constant $C>0$, such that for any $x \in [0,1]$, if $u_{0}>0$ we have
\[
\left| \hat u^{\eps}(x) - u_{0} - v \left( \frac{1-x}{\eps} \right) \right| \leq C e^{-\frac{f'(u_{0})}{\eps}},
\]
whereas if $u_{0}=0$ and $f$ satisfies (A1) we have
\[
\left| \hat u^{\eps}(x) - v \left( \frac{1-x}{\eps} \right) \right| \leq C \eps.
\]
\et
Note that in the characteristic case, we lose the exponential decay.

\begin{proof}
We use the expansion \eqref{WKB_1} and notice that $w$ satisfies
\begin{equation*}
\left\{
\begin{array}{l}
w''=(f'(u_{0}+v)-f'(u_{0}+v+w))v'-f'(u_{0}+v+w)w' \text{  ,  on } \left[0,\frac{1}{\eps}\right],\\
w(0)=0 \text{  ,  } w(\frac{1}{\eps})=-v(\frac{1}{\eps}).
\end{array}
\right.
\end{equation*}
Since $f'$ is increasing and $v'<0$, the maximum principle gives that $-v(\frac{1}{\eps}) \leq w \leq 0$. The result follows.
\end{proof}

\subsubsection{Single boundary layer at $x=0$}\label{singleBL0}
We study the case $u_{0}<u_{1} \leq 0$. The flow is outgoing at $x=0$ and we expect a boundary layer at $x=0$. Here we consider the following expansion, for any $x \in [0,1]$,
\[
\hat u^{\eps}(x) = u_{1} + v \left( \frac{x}{\eps} \right) + w \left( \frac{x}{\eps} \right)
\]
where $v$ is a function defined on $\R^{+}$ that satisfies the following ODE
\begin{equation}\label{ODE_BL0}
v'' = (f(u_{1}+v))' \text{   ,   } v(0)=u_{0}-u_{1} \text{  and  } v(+\infty)=0,
\end{equation}
and $w$ a function defined on $\left[0,\frac{1}{\eps}\right]$ supposed to be small and satisfying $w(0)=0$. We have similar results as in the previous subsubsection. 
\bl\label{v_BL0}
Assume that $f$ satisfies (A0). Let $u_{0},u_{1} \in \R$ such that $u_{0}<u_{1} \leq 0$. Assume that $v$ solves \eqref{ODE_BL0}. Then $v$ is negative, $v'$ is positive. There exists a constant $C>0$ such that, for any $y \in \R^{+}$, if $u_{1}<0$ we have
\[
|v(y)| \leq C e^{f'(u_{1})y},
\]
whereas if $u_{1}=0$ and $f$ satisfies (A1) we have
\[
|v(y)| \leq \frac{C}{1+y}.
\]
\el 

\bt
Assume that $f$ satisfies (A0). Let $\eps \in (0,1]$, $u_{0},u_{1} \in \R$ such that $u_{0}<u_{1} \leq 0$. Let $\hat u^{\eps}$ be the unique solution of Problem \eqref{steady_scalar} and $v$ the unique solution of Problem \eqref{ODE_BL0}. Then, there exists a constant $C>0$, such that for any $x \in [0,1]$,  if $u_{1}<0$ we have
\[
\left| \hat u^{\eps}(x) - u_{1} - v \left( \frac{x}{\eps} \right) \right| \leq C e^{\frac{f'(u_{1})}{\eps}},
\]
whereas if $u_{1}=0$ and $f$ satisfies (A1) we have
\[
\left| \hat u^{\eps}(x) - v \left( \frac{x}{\eps} \right) \right| \leq C \eps.
\]
\et

\subsubsection{Double boundary layer}

We finally study the case where $u_{0}<0<u_{1}$. Here the two endpoints are outgoing. We expect then a boundary layer on both side: a double boundary layer. We first need a technical lemma on the unique solution of \eqref{steady_scalar}.

\bl 
Assume that $f$ satisfies (A0) and (A1). Let $u_{0},u_{1} \in \R$ such that $u_{0}<0<u_{1}$. Assume that $\hat u^{\eps}$ is the unique solution of \eqref{steady_scalar}. Then, the exists a constant $C>0$ such that
\[
\left| \hat u^{\eps}\left( \tfrac12 \right) \right| \leq C \eps.
\]
\el 

\begin{proof}
From Assumptions (A0)-(A1), there exists two constants $\alpha,\beta>0$ so that for any $y \in (u_{0},u_{1})$, 
\begin{equation}\label{control_f}
\alpha |y|^{2} \leq f(y) \leq \beta |y|^{2}.
\end{equation}
Then, we introduce $x_{\eps}$ the unique element of $(0,1)$ such that $\hat u^{\eps} (x_{\eps}) = 0$. Integrating \eqref{steady_scalar}, we get
\[
\eps \hat u^{\eps}_{x} = \eps \hat u^{\eps}_{x}(x_{\eps}) + f(\hat u^{\eps}).
\]
Note also that $\hat u^{\eps}_{x}(x_{\eps}) = \min \hat u^{\eps}_{x}$ so that $u_{1}-u_{0} \geq \hat u^{\eps}_{x}(x_{\eps})>0$. Using \eqref{control_f}, we easily obtain from the previous ODE that
\[
\left( \arctan \left( \sqrt{\frac{\beta}{\eps \hat u^{\eps}_{x}(x_{\eps})} } \hat u^{\eps} \right) \right)' \leq \sqrt{\frac{\beta \hat u^{\eps}_{x}(x_{\eps})}{\eps}} 
\]
and integrating from $0$ to $1$ and using that $\eps \hat u^{\eps}_{x}(x_{\eps})$ tends to $0$, we prove the existence of $m>0$ such that that $\hat u^{\eps}_{x}(x_{\eps}) \geq m \eps$. Then, we similarly obtain
\[
\left( \arctan \left( \sqrt{\frac{\alpha}{\eps^2 m} } \hat u^{\eps} \right) \right)' \geq \sqrt{ \alpha m} 
\]
and integrating from $0$ to $\frac12$ and  from $\frac12$ to $1$ we get the desired estimate.
\end{proof}
Then, we introduce the two boundary layers: let $v_{0}$ function defined on $\R^{+}$ that satisfies the ODE
\begin{equation}\label{ODE_BL0d}
v_{0}'' = (f(v_{0}))' \text{   ,   } v_{0}(0)=u_{0} \text{  and  } v_{0}(+\infty)=0,
\end{equation}
and let $v_{1}$ function defined on $\R^{+}$ that satisfies the ODE
\begin{equation}\label{ODE_BL1d}
v_{1}'' = -(f(v_{1}))' \text{   ,   } v_{1}(0)=u_{1} \text{  and  } v_{1}(+\infty)=0.
\end{equation}
We can now state the main result of this subsubsection.

\bt
Assume that $f$ satisfies (A0) and (A1). Let $\eps \in (0,1]$, $u_{0},u_{1} \in \R$ such that $u_{0}<0<u_{1}$. Let $\hat u^{\eps}$ be the unique solution of Problem \eqref{steady_scalar} and $v_{0}$ and $v_{1}$  the unique solution of respectively Problem \eqref{ODE_BL0d} and Problem \eqref{ODE_BL1d}. Then, there exists a constant $C>0$, for any $x \in [0,1]$,
\[
\left| \hat u^{\eps}(x) - v_{0} \left( \frac{x}{\eps}\right) - v_{1} \left( \frac{1-x}{\eps} \right) \right| \leq C \eps.
\]
\et

\begin{proof}
Recall that Lemmas \ref{v_BL1} and \ref{v_BL0} shows that for any $y>0$
\[
0< -v_{0}(y) \lesssim \frac{1}{1+y} \text{   ,   } 0< v_{1}(y) \lesssim \frac{1}{1+y} \text{ ,  } v_{0}'>0 \text{  ,  } v_{1}'<0.
\]
We first consider what happens on $[0,\frac12]$. Decomposing $\hat u^{\eps}$ on $[0,\frac12]$ as
 \[
\hat u^{\eps}(x) = v_{0} \left( \frac{x}{\eps} \right) + w \left( \frac{x}{\eps} \right)
\]
we get, as in Subsection \ref{singleBL0} that
\begin{equation*}
\left\{
\begin{array}{l}
w''=(f'(v_{0}+w)-f'(v_{0}))v_{0}'+f'(v_{0}+w)w' \text{  ,  on } \left[0,\frac{1}{2\eps}\right],\\
w(0)=0 \text{  ,  } w(\frac{1}{2\eps})=\hat u^{\eps}(\tfrac12)-v_{0}(\tfrac{1}{2\eps}).
\end{array}
\right.
\end{equation*}
Similarly to the proof of Theorem \ref{thm_BL1}, we get $0 \leq |w| \leq |\hat u^{\eps}(\tfrac12)-v_{0}(\tfrac{1}{2\eps})|$ from the maximum principle. Using the previous lemma, the desired control on $[0,\frac12]$ follows since $\left| v_{0} \left( \tfrac{1}{2\eps} \right) \right| \lesssim \eps$ and for any $x \in [0,\frac12]$, $\left| v_{1} \left( \tfrac{1-x}{\eps} \right) \right| \leq \left| v_{1} \left( \tfrac{1}{2\eps} \right) \right|  \lesssim \eps$. A similar proof can be done on $[\frac12,1]$.
\end{proof}

\section{Asymptotic limits II : General case}\label{s:asymptoticII}

\subsection{Small-viscosity/large interval asymptotics}\label{s:smallvisc}
In either the vanishing-viscosity limit, or the large-interval limit $[0, X]$, $X\to +\infty$
after rescaling back to the unit interval $[0,1]$, we are led to consider in place of \eqref{conservation_law}
\begin{equation}\label{vvisc}
	\partial_{t} f^0(U) + f(U)_{x} = \eps \left(B(U) U_{x} \right)_{x} \text{ , } 0<x<1,
\end{equation}
with $\eps=\frac{1}{X}$, $\eps\to 0^+$, and the steady profile equation is
\be\label{vsteady}
f(U)' = \eps \left(B(U) U' \right)' .
\ee
Formally setting $\eps=0$ in \eqref{vsteady}, we obtain $f(U)\equiv \const$, or $U\equiv \const$ on smooth 
portions, separated by standing shock and boundary layers.
This indicates a rich ``zoo'' of possible steady solution structures.
Some examples from the isentropic gas dynamics case are displayed in Figure \ref{steadysolisentropic} which corresponds to behaviors we observed in Section \ref{s:asymptoticI}.

\begin{figure}[htbp]
\centering
\includegraphics[scale=0.13]{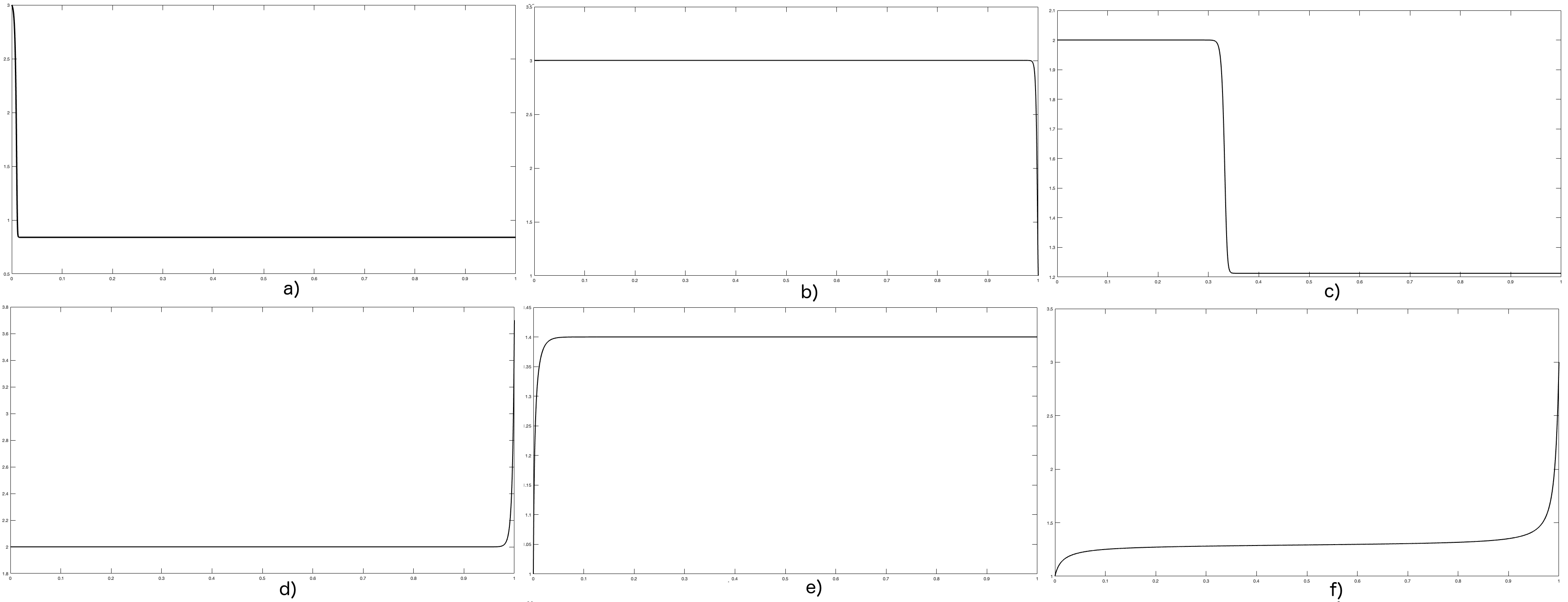}
\caption{Steady solutions of the Navier-Stokes isentropic equations \eqref{isogas} on $(0,1)$ with viscosity $\nu=0.01$ and $p(\rho)=\rho^{1.4}$. We only plot $u$.
Panels a) and b) depict left and right compressive boundary layers, c) an interior shock,
and d) and e) left and right expansive boundary layers, each connecting to nondegenerate
rest states of the steady profile ODE, hence exponentially convergent.  Panel f) depicts a double boundary layer (expansive)
consisting of left and right degenerate (hence non-exponentially convergent)
boundary layers meeting at a characteristic middle state.}
\label{steadysolisentropic}
\end{figure}

\subsubsection{Hyperbolic structure}\label{s:feasconfig}
It is readily deduced that the limiting configurations depicted in Figure \ref{steadysolisentropic} are in 
fact the only possibilities for the isentropic case.  For, $\rho u\equiv \const$ 
imposes $\rho, u>0$ throughout the limiting pattern, whence all states have either one or two positive characteristics
$\alpha=u\pm c$, where $c$ is sound speed. This in turn implies, by general results of
\cite{MaZ} that, as rest points of the scalar steady profile ODE, they are attractors or repellors, respectively.
A nontrivial nondegenerate boundary layer at the left endpoint $x=0$ must terminate at $x=0^+$ at a rest point, which must 
therefore be an attractor; at the right endpoint, $x=1^-$ a repellor.  Interior shocks must connect a repellor on the
left with a saddle on the right.  It follows that nontrivial boundary layers and interior shocks cannot coexist,
but occur only separately. Moreover, there can occur at most one nondegenerate boundary layer, either at the left or the
right endpoint.  The final possibility completing our zoo of possible configurations
is a double boundary layer configuration, for which the
end point must necessarilty be degenerate, corresponding to a ``sonic'', or ``characteristic'' point where $u=c$.
Shocks or boundary layers connecting to a nondegenerate rest point decay exponentially; those connecting to
a degenerate rest point decay algebraically. For further discussion of boundary layer structure for the compressible
Navier-Stokes equations, see, e.g., \cite{SZ,GMWZ}.

For the nonisentropic case, we again have $\rho, u>0$, imposing in this instance that state have either two
or three positive characteristics $\alpha= u-c, u, u+c$;
as rest points of the steady profile equation, these correspond to saddle points or repellors, respectively.
Similar analysis to the above yields again that left boundary layers and shocks cannot coincide; however, there
is the new possibility of patterns consisting of a left boundary layer plus a right boundary layer, or
an interior shock plus a right boundary layer, as nontrivial right boundary layers may 
connect to either repellors or saddles in the nonisentropic case. 

\subsubsection{Rigorous asymptotics}\label{s:rig}
A very interesting open problem is to carry out the zero-viscosity limit rigorously, in preparation 
for the more complicated dynamics of the 2d shock tube problem.
One might hope also to understand the spectra of such wave patterns 
as the approximate direct sum of the spectra of component layers, as would follow, for example, by the methods of
\cite{Z1,Z4} if the components layers remained appropriately spatially separated in the limit.
A first apparently nontrival step, of interest in its own right, is to show for given boundary data
existence and uniqueness of feasible limiting patterns as described in Section \ref{s:feasconfig}.

We illustrate this for the (surprisingly complicated) case of isentropic gas dynamics:
\be\label{isogas}
\begin{aligned}
&\rho_t + (\rho u)_x=0,\qquad (\rho u)_t + (\rho u^2 + p(\rho))_x= \nu u_{xx}; \qquad \rho>0 , \qquad p', p''>0.\\
&\rho(0)=\rho_{0} \text{  ,  } u(0)=u_{0} \text{  ,  } u(1)=u_{1}.
\end{aligned}
\ee

\begin{proposition}\label{iso_zoo}
	For isentropic gas dynamics \eqref{isogas}, there exists, for each boundary data $(\rho_{0},u_{0},u_{1})$, a unique limiting steady configuration 
	consisting of: (i) a single boundary layer at left or right endpoint; (ii) a single interior
	standing shock; or (iii) a double boundary-layer connecting from each boundary to a ``characteristic 
	point'' $u=c$, with $c=\sqrt{p'(\rho)}$ denoting sound speed.
	These determine a unique inviscid solution up to location of the interior shock.
	The unique viscous steady solution converges to an
	inviscid one with particular
	interior shock location as viscosity $\nu\to 0$
	at sharp rates $\nu^{1/p}$ in $L^p$, $1<p<\infty$ and $\nu\log\nu$ in $L^1$.
\end{proposition}

\begin{proof}
	We fix $(\rho_{0},u_{0})$. Referring to Section 2 in \cite{MelZ}, a steady solution $(\hat \rho ,\hat u)$ of \eqref{isogas} is given by 
	\be\label{isoprof}
	m\nu \hat \rho'= \hat \rho^2(b- \psi(\hat \rho)), \qquad \hat u= m/\hat \rho,
	\ee
	where $\nu$ is coefficient of viscosity;  $m= \rho_{0} u_{0}>0$;
	$\psi(\rho):= m^2/\rho + p(\rho)$ is convex, with limit $+\infty$ as $\rho\to 0$ and $\rho\to +\infty$; 
	and $b$ is a constant determined implicitly by $\hat \rho(1)=m/u_{1}$.

	Evidently, $\psi$ is minimized at the characteristic point $\rho_*$ where $m^2/\rho_{*}^2=p'(\rho_{*})$. 
	Note that by denoting $u_{*}=m/\rho_{*}$ we get $u_{*}=c$ by positivity of $u_{*}$ and $c$.
	For $b>\psi_*:=\psi(\rho_*)$, there exist nondegenerate rest points $\rho_-<\rho_*<\rho_+$, with $\hat \rho'>0$
	for $\rho_-<\hat \rho<\rho_+$ and $\hat \rho'<0$ for $\hat \rho<\rho_-$ and $\hat \rho>\rho_+$.
	For $b=\psi_*$, there is a single degenerate rest point at $\hat \rho=\rho_*$, with $\hat \rho'<0$ for
	all $\hat \rho\neq \rho_*$. Checking $\psi''(\rho_*)=2m^2/\rho_*^3+p''(\rho_*)>0$, hence to lowest
	order the flow near $\rho_*$ is of Riccati type, $\hat \rho'\sim - \nu^{-1}\hat \rho^2$.
	For $b<\psi_*$, there are no rest points and $\hat \rho'$ is uniformly negative.
	See Figure \ref{figphase} for a typical phase portrait.

   \begin{figure}[htbp]
   \centering
   \begin{tikzpicture}
 \draw[line width=0.5mm,color=black] plot [domain=0:4] (\x,0.3*\x^2); %right curve
  \draw[line width=0.5mm,color=black] plot [domain=-3.5:0] (\x,-0.4*\x^2);%left curve
  \draw[line width=0.4mm,color=black][->] (-4,-1.5) -- (-4,5.5); % ordinate axis
  \draw[line width=0.4mm,color=black][->] (-5,-1) -- (5,-1); %abscissa axis
  
  \draw (4.1,4.3) node{\Large $\psi$}; % \psi
  
  \draw[line width=0.4mm,color=black][-] [dashed] (-4,3) -- (4,3); %vertical line b
  \draw (-4.2,3) node{\Large $b$}; % b

  \draw[line width=0.4mm,color=black][-] [dashed] (3.162277660,3) -- (3.162277660,-1); %vertical line rho_2 abscissa value : sqrt(3/0.3)
  \draw (3.162277660,-1.3) node{\Large $\rho_{+}$}; % rho_2
  
  \draw[line width=0.4mm,color=black][-] [dashed] (-2.73861278753,3) -- (-2.73861278753,-1); %vertical line rho_1 abscissa value : -sqrt(3/0.4)
  \draw (-2.73861278753,-1.3) node{\Large $\rho_{-}$}; % rho_1
  
   \draw[line width=0.4mm,color=black][-] [dashed] (0,0) -- (0,-1); %vertical line rho_gamma
  \draw (0,-1.3) node{\Large $\rho_{\ast}$}; % rho_ast
  
  \draw[line width=0.4mm,color=black][->] (-2.9,-0.7) -- (-4,-0.7); %left arrow toward the left
   \draw[line width=0.4mm,color=black][->] (-2.6,-0.7) -- (3,-0.7); %middle arrow toward the right
   \draw[line width=0.4mm,color=black][->] (5,-0.7) -- (3.3,-0.7); %right arrow toward the left
\end{tikzpicture}
  \caption{ Phase portrait of the ODE \eqref{isoprof}}
  \label{figphase} 
   \end{figure}
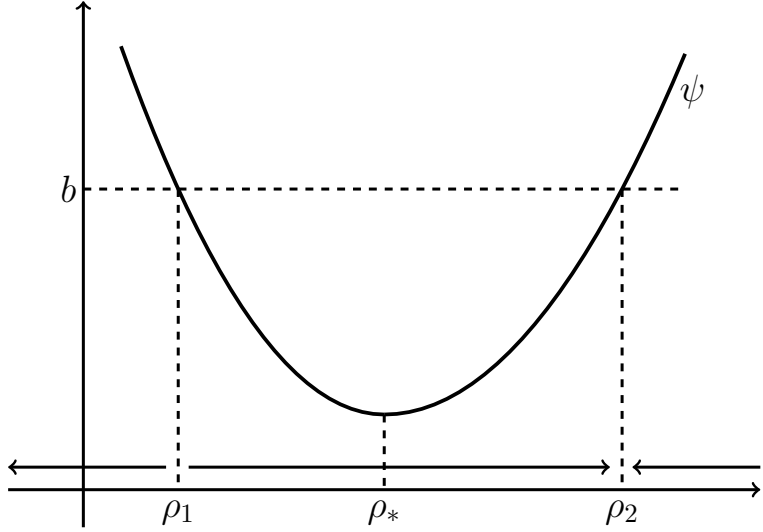

	{\bf (case $\rho(0)\geq \rho(1)$)}
	It follows that for any $\rho(0)>\rho(1)>\rho_*$, there exists a unique decreasing right viscous boundary layer 
	with value $\rho(1)$ at $x=1$ and converging at exponential rate $e^{-r_1 |x|/\nu}$ as
	$x\to -\infty$ to the lefthand state $\rho(0)$. Similarly, for any $\rho_*>\rho(0)>\rho(1)$, there
	is a unique decreasing left viscous boundary layer with value $\rho(0)$ at $x=0$ and
	converging as $e^{-r_0 |x|/\nu}$ as $x\to + \infty$ to $\rho(1)$.
	For $\rho(0)\geq \rho_*\geq \rho(1)$, there are unique decreasing {\it characteristic} 
	right and left viscous boundary layers connecting $\rho(1)$ and $\rho(0)$ to the characteristic point $\rho_*$, 
	with convergence at {\it algebraic rate} $|x/\nu|^{-1}$.  

	{\bf (case $\rho(0)\leq \rho(1)$)}
	Similarly, for $\rho(0)<\rho(1)<\rho_*$, there is a unique increasing and exponentially
	converging right viscous boundary layer from $\rho(1)<\rho_+(b)$ to $\rho(0)=\rho_-(b)$ for some $b>\psi_*$,
	while for $\rho_*< \rho(0)<\rho(1)$, there is a unique increasing and exponentially
	converging left viscous boundary layer from $ \rho(0)>\rho_-(b)$ to $\rho(1)=\rho_+(b)$ for some $b>\psi_*$.
	For $\rho(0)<\rho_*<\rho(1)$, finally, there exists either a unique
	increasing and exponentially converging right viscous boundary layer from $\rho(1)<\rho_+(b)$ 
	to $\rho(0)=\rho_-(b)$ for some $b>\psi_*$,
	a unique increasing and exponentially converging left viscous boundary layer from $ \rho(0)>\rho_-(b)$ 
	to $\rho(1)=\rho_+(b)$ for some $b>\psi_*$,
	or an exponentially converging stationary interior shock from $\rho(0)=\rho_-(b)$ to $\rho(1)=\rho_+(b)$
	for some $b>\psi_*$.

	Note, at the hyperbolic (inviscid, $\nu=0$) level, 
	that left and right characteristic boundary layers may be concatenated
	to form a ``double boundary layer'' with intermediate characteristic state $\rho_*$, connecting
	$\rho(0)>\rho_*$ to $\rho(1)<\rho_*$.  In all other cases, the data $\rho(0)$ and $\rho(1)$
	may be connected by a single element consisting of a left or right, increasing or decreasing boudary layer,
	or, in the limiting case, a single stationary interior shock.

	Combining these elements, we obtain a unique feasible inviscid limit, up to arbitrary positioning of the
	interior standing shock. 
	A refined analysis taking into account exponential decay rates $e^{-r_0|x|/\nu}$ and
	$e^{-r_1|x|/\nu}$ of the viscous shock as $x\to -\infty$ and $x\to +\infty$, respectively,
	combined with the linearization $[(\bar u^2-\bar c^2)\Delta \rho]=0$ 
	of Rankine-Hugoniot relation $[m^2/\rho+ p(\rho)]=0$, where $[\cdot]$ denotes jump across the shock
	and $\Delta \rho$ denotes change in $\rho$ due to truncation of the (infinite-extent) viscous shock
	at boundaries $x=0,1$, we obtain asymptotics 
	$r_0 x_s= r_1(1-x_s)+ O(\nu)$, or
	\be\label{shockloc}
	 x_s(\nu)= \frac{r_1}{r_0 +r_1}  + O(\nu)
	\ee
	as $\nu\to 0$ toward the limiting shock location $x_s(0)=r_0/(r_0+r_1)$.  Similarly,
	the single left- or right-boundary layer solutions must be adjusted slightly due to exponential
	tails, in order to reach the correct value at the finite endpoint of the interval.
	We omit the details.

	Putting these pieces back together, we obtain convergence of viscous to inviscid solution
	in $L^1$ at rate $O(\nu)$ for solutions containing shocks and noncharacteristic boundary layers.
	The analysis of double layer solutions formed by concatenation of characeristic left and right boundary
	layers is somewhat more involved.
	Noting that $\psi\sim (\rho-\rho_*)^2$ near $\rho_*$ in the boundary layer profile ODE,
	and splitting into regions where $(\rho-\rho_*)^2 \lesssim b$ and 
	$(\rho-\rho_*)^2 \gtrsim b$, we find after a brief calculation that $|\rho-\rho_*|$ varies 
	from $0$ to $\sqrt{b}\sim \nu$ on an $x$-interval of order-one length and afterward grows like $\nu/x$,
	giving a total $L^1$ convergence error $\sim \nu + \int_\nu^1 (\nu/x)dx \sim \nu \log \nu $.
	Other $L^p$ norms, $1<p<\infty$, go similarly.
\end{proof}

\subsubsection{Hyperbolic dynamics}\label{s:hypdesc}
A very interesting further question is the hyperbolic description of {\it dynamics} in the inviscid limit,
in particular whether one may extend the formal description of Gisclon--Serre \cite{GS94,G96,GMWZ10}
for the half-line $[0,+\infty)$, based on viscous profiles, to the case of a bounded interval. 
In this approach, one imposes on the interior of the set the usual hyperbolic description, with jump
conditions given by the requirement that there exist a corresponding viscous shock profie connecting the
endstates of the jump.  Analogously, one requires at the boundary that solutions lie in the ``reachable set''
of limiting states as $x\to +\infty$ of a viscous boundary layer profile
satisfying the viscous boundary conditions at $x=0$. For {\it noncharacteristic} boundary layers,
this description may be rigorously validated under very general circumstances \cite{G96,GMWZ10},
both in the sense that it provides a well-posed hyperbolic problem, and that the solution of this problem
is the limit of viscous solutions. 

In the present, bounded interval case, the appearance of characteristic boundary layers complicates matters.
For illustration, considering again the case of isentropic gas dynamics \eqref{isogas} with viscous
inflow/outflow boundary conditions specifying 
\be\label{vBC}
(\hat{\rho}(0), \hat{u}(0), \hat{u}(1))=(\rho_0, u_0, u_1); \qquad \rho_0, u_0, u_1 >0.
\ee

Then, the classification of boundary layers in the
proof of Proposition \ref{iso_zoo} shows that the left inviscid boundary conditions induced by ``reachability''
criterion of Gisclon-Serre are given for ``supersonic'' states $\rho(0) \leq \rho_*$ by the 
full Dirichlet conditions
\be\label{superlBC}
(\rho(0),  u(0))= (\rho_0, u_0), \qquad \rho_0\leq \rho_*,
\ee
and for ``subsonic'' states $u(0)\leq \sqrt{p'(\rho(0))}$ by 
\be\label{sublBC}
\rho(0) u(0)= m_0:= \rho_0 u_0, \qquad \rho(0) \geq \rho_0^\dagger,
\ee
where $\rho^\dagger$ denotes the conjugate point connected to 
$\rho$ by a standing viscous shock profile with $\rho u=m_0$.
Here, $\rho(0),u(0),u(1)$ refer to interior hyperbolic solutions.

Different from the small-amplitude case treated in \cite{GS94},
both of these conditions are on a restricted range depending on $\rho_0$, outside of which states 
are {\it disallowed}.
Note, if $\rho_0$ and $\rho(0)$ are nearby and far from $\rho_*$, then $\rho(0)<\rho_*$ implies $\rho_0<\rho_*$, 
necessarily.
Similarly, $\rho(0)$ subsonic implies $\rho_0>\rho_*$, hence $\rho_0^\dagger \leq \rho_*$ and so
$\rho(0) \geq \rho_* \geq \rho_0^\dagger$ always. Thus, the restrictions in 
\eqref{superlBC}(ii) and \eqref{sublBC}(ii)
are only for large-amplitude perturbations, so long as all states are bounded away from characteristic values.

Meanwhile, for subsonic states $u(1) \leq \sqrt{p'(\rho(1))}$ (which serve only as repellors in backward $x$ flow, so
admit only trivial, constant boundary-layer connections), 
the induced right inviscid boundary conditions are given, setting $m_0=\rho(1)u(1)$, by
\be\label{subrBC}
u(1)= u_1 \, (>0), 
\ee
in agreement with the number of incoming (hyperbolic) characteristic modes to the domain.
For supersonic states $u(1)\geq \sqrt{p'(\rho(1))}$, for which all characteristic modes leave
the domain, the state $(\rho(1), u(1))$ is an attractor in backward $x$ for the boundary layer profile ODE
with appropriate constant $b$, connecting to all states $\rho_1 \leq \rho(1)^\dagger$.
Thus, there are {\it no imposed right boundary conditions}, in agreement with
the number (zero) of incoming modes, other than the ``range'' condition
\be\label{superrBC}
\rho_1:=m_1/u_1 \leq \rho(1)^\dagger, \qquad m_1:=\rho(1)u(1)
\ee
on the set of allowable $u(1)$. This is always satisfied for $u(1)\leq u_1$, but for $u(1)>u_1$
imposes a strictly stronger upper bound on $\rho(1)$ than subsonicity.

Here, we have considered in our discussion of right boundary conditions 
only states $u(1)>0$; the treatment of negative values would involve also
``$2$-family'' boundary layers with $m<0$, hence additional complexity.

As all conditions are open ones,
for solutions lying near a {\it noncharacteristic} steady solution,
the operant boundary conditions are of the standard type considered in hyperbolic initial boundary value theory,
hence in this local setting {\it the formal description of Gisclon-Serre may indeed be extended sensibly to the
case of a bounded interval,} as follows also by the original observations of \cite{GS94,G96,GMWZ10}.
However, the range-restrictions on these boundary conditions are not
not consistent with standard noncharacteristic hyperbolic theory, but rather a new ``obstacle'' type 
boundary condition.
Thus, both near the characteristic double layer type solutions, for which \eqref{sublBC} may enter, or
in the gobal, large-amplitude setting where \eqref{superrBC} comes into play, 
it seems a very interesting open problem whether such a boundary condition
can determine a reasonable hyperbolic flow.
We emphasize for the case of double layer solutions that {\it even small perturbations of the background
steady state lead to consideration of nonstandard boundary condition \eqref{sublBC}}. 
This might be called ``transcharacteristic-type,'' bridging as it does between sub- and supersonic boundary
conditions, and seems one of the more interesting aspects of the investigations of this section.

More generally, the questions of global existence for large-amplitude data for the viscous problem on 
the interval, and convergence to a vanishing viscosity limit seem very interesting,
especially in the isentropic case where the entropy-based method of {\it compensated compactness} 
has so been successful on the whole-line problem \cite{CP10,D83,S86}.
One may ask in particular whether or not boundary entropy conditions as in \cite{DL88} might agree 
with the standard and nonstandard boundary conditions derived by viscous profile considerations above.

\medskip

{\bf The left boundary Riemann problem.}
Let us go a bit further in consideration of the transcharacteristic-type boundary conditions
\eqref{sublBC}--\eqref{superrBC},
and their role in existence/well-posedness of hyperbolic solutions near a constant characteristic
solution $(\bar \rho, \bar u)\equiv (\rho_*,u_*)$, $u_*=\sqrt{\rho_*}$.
Recall that linearized and nonlinear stability hinge upon solvability of the {\it boundary Riemann problems}
at left and right boundary of the interval, i.e., solvability on the half-line of the problem with constant initial
data together with the given boundary conditions.  For definiteness, let us first discuss the left boundary,
renoting the constant data to the right of the boundary as $(\rho_R, u_R)$. Then, the solution if it exists
consists of a left state $W_L=(\rho_L, u_L)$ next to the boundary and satisfying the boundary conditions, connected
by a (possibly trivial, or zero-strength) $1$-shock or -rarefaction to a middle state $(\rho_M, u_M)$, 
which is in turn connected by a $2$-shock or -rarefaction to $(\rho_R,u_R)$.

We now reverse the point of view of the previous discusion and fix boundary data $\rho_0, u_0$, asking in turn
what states $(\rho_L, u_L)$, $(\rho_M, u_M)$ and $(\rho_R, u_R)$ it may be connected to.
The set of connectible states $(\rho_L, u_L)$ is exactly the admissible set $\mathcal{E}$ defined in \cite{GS94}.
We define the set $\mathcal{E}'$ to be the ``composite'' admissible set of states $(\rho_M,u_M)$
reachable by a combination of boundary-layer and $1$-wave solutions.

(i) {\it Subsonic case ($\rho_0\geq \rho_*$). }
A key observation is to notice that for any subsonic states $(\rho_L, u_L)$
lying on the admissible boundary-layer curve $\rho u=m_0=\rho_0u_0$, the $1$-characteristic has strictly negative
speed leading out of the domain, hence there are no admissible $1$-waves in associated boundary Riemann patterns.
On the other hand, at the extreme, characteristic point $(\rho_L, u_L)=(\rho_*,u_*)$ bounding the admissible
set, $u_*=\sqrt{P'(\rho_*)}$, the $1$-characteristic speed is zero, and thus one may adjoin to this state 
a $1$-rarefaction wave moving to the right into the domain, since all characteristics in the Riemann fan will have
speed $\geq 0$. One cannot adjoin a $1$-shock however, as its speed would be strictly less than that of the left-most
$1$-characteristic speed, which would be zero.
Thus, the set of admissible middle states $(\rho_M,u_M)$ consists precisely of portion of the admissible curve
$\rho u=m_0$ with $\rho \geq \rho_*$, {\it concatenated with the $1$-rarefaction curve emanating from $\rho_*$,} 
also pointing in the direction of decreasing $\rho$.

\bl\label{concatlem}
For $P', P''>0$, and $\rho_0>\rho_*$, the composite admissibility set $\mathcal{E}'$ is a globally $C^1$ decreasing
graph of $u$ over $\rho\in \R^+$, 
tangent at $(\rho_*,u_*)$ to the $1$-Hugoniot curve through $(\rho_*,u_*)$, and transversal at every point to
the $2$-Hugoniot curve originating from that point.
\el

\begin{proof}
	It is a standard fact \cite{Sm} that the $1$-rarefaction curve through $(\rho_*,u_*)$
	is tangent to the $1$-Hugoniot curve through $(\rho_*,u_*)$, consisting of states $(\rho, u)$
	connected to $(\rho_*,u_*)$ by Rankine--Hugoniot conditions 
	\be\label{RH}
	s(\rho- \rho_*)= \rho u - \rho_*u_*,\quad
	s(\rho u - \rho_*u_*)= \frac12 (\rho u^2- \rho_* u_*^2 + P(\rho)- P(\rho_*),
	\ee
	where $s$ is the associated shock speed.
	For isentropic gas dynamics under the assumptions $P',P''>0$, the solution may be expressed \cite{Sm} as 
	global $C^1$ functions $u(\rho)$, $s(\rho)$,
	with $s(\rho_*)$ equal to the characteristic speed at $\rho_*$, or zero.
	Thus, in the vicinity of $\rho_*$, $s(\rho-\rho_*)= O(|\rho-\rho_*|^2)$, and so, rearranging
	\eqref{RH}(i), we have $\rho u = m_0 + O(|\rho-\rho_*|^2)$, or
	$$
	 u = \frac{m_0}{\rho} + O(|\rho-\rho_*|^2), 
	 $$
	 giving the asserted tangency at $\rho_*$ to the boundary admissibility curve $u=m_0/\rho$.

	 Transversality follows, similarly as for transversality of $1$- and $2$-Hugoniot curves in the study
	 of the initial value Riemann problem \cite{Sm} from the fact that both are graphs over $\rho$,
	 with $u$ strictly decreasing on the first and increasing for the second.
\end{proof}

(ii) {\it Supersonic case ($\rho_0\leq \rho_*$). }

\bl\label{concatlemii}
For $P', P''>0$, and $\rho_0\leq \rho_*$, the composite admissibility set $\mathcal{E}'$ is decreasing,
$C^1$ on $\rho \gtrless \rho_0^\dagger$, continuous at $\rho=\rho_0^\dagger$,
and transversal at every point to the $2$-Hugoniot curve originating from that point.
\el

\begin{proof}
Similarly as in the previous case, we have that $\mathcal{E}$ consists of the single supersonic state
	$(\rho_L,u_L)=(\rho_0, u_0)$ together with the subsonic states $\rho_L\geq \rho_1^\dagger$;
in particular, it is disconnected.
	However, $1$-shocks from $(\rho_0,u_0)$ to $(\rho,u)$ have positive speed, hence are admissible in
	a boundary Riemann pattern, for $\rho\leq \rho_0^\dagger$, with boundary $\rho_0^\dagger$
	representing the point where shock speed is exactly zero. But, this point joins continuously to $\mathcal{E}$,
	giving continuity.  Moving in the other direction along the $1$-Hugoniot curve, with $\rho\leq \rho_0$,
	gives right-moving rarefaction waves, also admissible in boundary rarefaction patterns.  Thus, the
	$1$-Hugoniot curve for $\rho\leq \rho^\dagger$ joined with $\mathcal{E}$ gives the described curve
	$\mathcal{E}'$.
	One may check that the two curves are {\it not} tangent at $\rho=\rho_0^\dagger$, however, as
	the shock curve satisfies at the characteristic point $s=0$ the derivative
	formula 
	$\dot m = \dot s(\rho-\rho_0)\neq 0$, whereas the admissibility curve satisfies $\dot m\equiv 0$,
	where $m=\rho u$.
\end{proof}

\bc\label{BCLcor}
Assuming $P',P''>0$, for any right state $(\rho_R,u_R)$, there is a unique middle state 
$(\rho_M,u_M)=\Phi(\rho_R,u_R)$
lying on the composite admissibility curve $\mathcal{E}'$, hence
the left boundary Riemann problem has a unique \emph{global} solution.
Moreover, $\Phi$ is piecewise $C^1$, with discontinuous derivative along a single curve.
\ec

\begin{proof}
	From the fact that the $2$-Hugoniot curve $u_2(\rho)$ through any $(\rho_R,u_R)$ 
	is strictly increasing for $\rho \in \R^+$,
	with limits $0$ and $+\infty$ as $\rho\to 0$ and $\rho \to +\infty$, and the fact that the graph of 
	$\mathcal{E}' $ is strictly decreasing, with limits $u\to +\infty$ and $u\to 0$ as $\rho \to 0$,
	$\rho \to +\infty$, we find that there is a unique intersection between the $2$-Hugoniot curve 
	and $\mathcal{E}'$. This gives existence and uniqueness; $C^1$ depedence then follows by the Inverse
	Function Theorem by $C^1$ regularity and transversality 
	of the graphs of $\mathcal{E}$ and the $2$-Hugoniot curve, everywhere except on the forward $2$-Hugoniot
	curve emanating from $(\rho_0^\dagger, m_0/\rho_0^\dagger)$.
\end{proof}

\medskip

{\bf The right boundary Riemann problem.}
We now study the right boundary problem, with data $u_1$. To match the previous discussion, we will treat 
$\rho_1$ as a fixed parameter, and define for each $\rho_1$ the admissible set $\mathcal{F}$ of
states $(\rho_R,u_R)$ connectible to $(\rho_1, u_1)$ by a  boundary layer. As in the previous case, we will
take all velocities $u$ to be positive. Thus, $2$-waves move to the right out of the domain and cannot
be part of any boundary Riemann pattern at $x=1$. We define the composite admissible set $\mathcal{F}'$
to be the set of states $(\rho_L,u_L)$ reachable by a (possibly trivial) 
leftgoing $1$-shock or -rarefaction from $\mathcal{F}$.

(i) {\it Subsonic case ($\rho_1 \geq \rho_*$). }
This case is analogous to left boundary case (ii).
For the right boundary layer problem, we consider backward flow in the boundary-layer ODE, for which 
subsonic equilibria are repellors, and supersonic equilibria attractors. Thus, nearby $(\rho_1,u_1)$,
the only boundary-layer connection is the trivial one, to itself, or $(\rho_R,u_R)=(\rho_1,u_1)$.
On the other hand, $(\rho_1,u_1)$ may be connected to any supersonic state $(\rho_R,u_R)$ with
$\rho_R\leq \rho_1^\dagger$, with the limiting point $\rho_R=\rho_1^\dagger$ corresponding to
a zero-speed $1$-shock. Thus, $\mathcal{F}$ consists of the disconnected set of points on the graph
$\rho u=m_1:=\rho_1u_1$ lying over $\rho_R=\rho_1$ and $\rho_R\leq \rho_1^\dagger$.

There are no preceding $1$-waves possible for the supersonic states $\rho_L\leq \rho_1^\dagger$, as these would be
right-moving out of the domain. The subsonic state $(\rho_1,u_1)$ however, admits a preceding $1$-rarefaction from
any point $(\rho_L,u_L)$ on the $1$-Hugoniot curve through $(\rho_1,u_1)$ with $\rho_L \geq \rho_1$, and
a preceding $1$-shock from any point $(\rho_L,u_L)$ on the $1$-Hugoniot curve through $(\rho_1,u_1)$ with 
$\rho_L\geq \rho_1^\dagger$, terminating at the point $\rho_L= \rho_1^\dagger$ corresponding to a zero-speed
$1$-shock.

\bl\label{Rconcatlemi}
For $P', P''>0$, $u_1>0$, and $\rho_1\geq \rho_*$, the composite admissibility set $\mathcal{F}'$ is decreasing,
$C^1$ on $\rho \gtrless \rho_1^\dagger$, continuous at $\rho_0^\dagger$ to the $1$-Hugoniot curve 
through $(\rho_0,u_0)$, and transversal at every point to the $2$-Hugoniot curve originating from that point.
\el

\begin{proof}
Essentially identical to the proof of Lemma \ref{concatlemii}.
\end{proof}

(ii) {\it Supersonic case ($\rho_1\leq \rho_*$). } 
In this case, which is analogous to left boundary case (i), 
$(\rho_1,u_1)$ may be connected by boundary layer to any supersonic state, i.e., to any
state $(\rho_R,u_R)$ on the curve $\rho u\equiv m_1$ with $\rho_R\leq \rho_*$, 
with the limiting point $\rho_R=\rho_*$ corresponding to a characteristic boundary layer.
The set $\mathcal{F}$ thus consists of all such supersonic states.
From strictly supersonic states, no preceding $1$-wave is possible, since it would be right-moving out
of the domain. However, the characteristic boundary point $\rho_R=\rho_*$ may be preceded by a $1$-rarefaction
from any point $\rho_L\geq \rho_*$.  The composite admissible set $\mathcal{F}'$ thus consists of the
concatenation of the boundary-layer curve up to $\rho=\rho_*$ with the $1$-rarefaction curve emanating
from $(\rho_*, m_1/\rho_*)$.

\bl\label{Rconcatlemii}
For $P', P''>0$, $u_1>0$, and $\rho_1\leq \rho_*$, the composite admissibility set $\mathcal{F}'$ is decreasing,
globally $C^1$, and tangent at $\rho=\rho_*$ to the $1$-Hugoniot curve through $(\rho_*, m_1/\rho_*)$.
\el

\begin{proof}
Essentially identical to the proof of Lemma \ref{concatlem}.
\end{proof}

\bc\label{BCRcor}
Assuming $P',P''>0$, $u_1>0$, any left state $(\rho_L,u_L)$, $u_L>0$, lies on
lying on the composite admissibility curve $\mathcal{F}'$ for a unique boundary value
$$
(\rho_1,u_1)= (\phi(\rho_L,u_L,u_1),u_1).
$$
Moreover, $\Phi$ is piecewise $C^1$ with discontinuous derivative along a 2D surface in $(\rho_L,u_L,u_1)$.
\ec

\begin{proof}
	By Galillean invariance, the $1$-Hugoniot curves through $(\rho_1,u_1)$ are translates of each other
	in $u$ and cover the $\rho$-$u$ plane in one-to-one fashion
	as $u_1$ is varied. Likewise, the curves $\rho u= m_1$ cover
	the plane in one-to-one fashion as $m_1$ is varied. The case (i) cuts these families along the line
	$\rho=\rho_1^\dagger$ and glues left and right halves together; moreover, the cutoff curve may be seen
	to intersect curves in each family precisely once. At the limiting case $\rho_1=\rho_*$, this cutoff is
	precisely at $\rho_*$.
	Case (ii) is a similar gluing of the same two families of curves, but this time cut at $\rho_*$,
	again with a single intersection of cutting curve with individual curves in each family.
	Thus, the two halves together cover the plane in one-to-one fashion, giving a unique solution of
	the right boundary Riemann solution.
	By the Implicit Function Theorem, moreover, this is $C^1$ except where $(\rho_1,u_1)$
	is subsonic, with $\rho_1u_1= m_1=\rho_L u_L$ and $\rho_1^\dagger = \rho_L$, tracing out a single
	curve with one boundary of discontinuity in the $(\rho_L, u_L)$ plane, with vertex at the characteristic
point $\rho_L u_L/u_1= \rho_*$.
Taking the union over $u_1$, we are done.
\end{proof}

\br\label{comrmk}
The constructions of Lemma \ref{concatlem}-\ref{Rconcatlemii} are 
reminiscent of composite wave constructions in initial value Riemann solutions for nonconvex conservation laws
\cite{Da,L}.
\er

{\bf Well-posedness near steady solutions.}
Reviewing Corollaries \ref{BCLcor} and \ref{BCRcor} we see that small perturbations of steady hyperbolic solutions
lead to boundary Riemann problems in the $C^1$ part of the solution operator. In particular, in the new case
of the double characteristic boundary-layer solution, we have $(\rho_0,u_0)$ subsonic and
$(\rho_1, u_1)$ supersonic, with middle state $(\hat \rho, \hat u)\equiv (\rho_*,u_*)$ characteristic, 
where $\rho u\equiv m_0= \rho_0u_0$ for all states. Thus, small perturbations of the interior solution
lead to boundary Riemann problems in the $C^1$ cases described in Lemmas \ref{concatlem} and \ref{Rconcatlemii}.

This yields well-posedness in the Sobolev sense of Kreiss and Majda, as considered in \cite{GS94} of small
perturbations of steady hyperbolic solutions of the inflow/outflow problem.
Behavior for global data remains a very interesting open problem, as it is also in more standard cases.
The effect of discontinuous derivatives in the boundary Riemann solver for small data solutions is another
very interesting open question.

\medskip
{\bf Entropy boundary conditions.} In \cite{DL88} there is derived for general conservation 
laws \eqref{conservation_law} the (left) entropy boundary inequality
\be\label{eineq}
q(U(0))- q(U_0)\leq d \eta(U_0)(U(0)-U_0),
\ee
for any entropy/entropy flux pair $\eta$ (not necessarily convex)/$q$ such that $\Re(d^2\eta(U) B(U))\geq 0$,
also called ``$B$-convex.''
A symmetric inequality holds for the right boundary.
As shown in \cite[Prop. 2]{GS94}, 
the Gisclon-Serre type boundary condition obtained from viscous boundary-layer criteria
necessarily satisfies \eqref{eineq}.
In \cite[Prop. 1.1]{DL88}, it is shown in the linear case that conditions \eqref{eineq} agree with a formulation 
in terms of Riemann problems that enforces unique solutions, hence the two concepts must agree.
However, in the general nonlinear case they clearly differ, as there does not necessarily exist any 
$B$-convex entropy.  For the case of isentropic gas dynamics it was conjectured in \cite{DL88} that they agree, but
to our knowledge this question has up to now neither been proved nor disproved.

Equivalence of the two types of boundary condition would be extremely interesting, as it would imply that {\it any} 
weak vanishing viscosity solution would satisfy Gisclon-Serre type boundary conditions, 
whereas in \cite[Prop. 1]{GS94} it is shown that smooth (noncharacteristic) Gisclon-Serre type solutions are vanishing viscosity limits of  {\it some} sequence of solutions.

\subsection{The standing shock limit}\label{s:standing}
A simple case in which the zero-viscosity limit can be completely carried out is that
of the ``standing-shock limit'' generalizing the study of \cite{Z1} in the case of the half-line.
This consists of the study of a stationary viscous $n$-shock $\hat U^\eps(x)=\bar U \left( \frac{1}{\eps} (x-\frac{1}{2}) \right)$ of \eqref{vvisc},
solving \eqref{vsteady} for all $\eps>0$, with respect to its ``own'' boundary conditions, i.e.
$$
U_0=\hat U^\eps(0),\quad U_1^{II}=\hat U^\eps(1).
$$
We consider this for the general class of system \eqref{conservation_law},\eqref{BC2} under assumptions (\textbf{H0})-(\textbf{H3}) plus the additional assumption used in \cite{Z1}:

\medskip
\medskip
\noindent (\textbf{H4}) the eigenvalues of $df(U_{0})$ and $df(U_{1})$ are nonzero.\footnote{By (\textbf{H3}) the eigenvalues of $df(U)$ are real and semi-simple since $df(U)$ is symmetric for the inner product associated to $S(U)$.}
\medskip
\medskip

Converting by $x\to \frac{x}{\eps}$ to the large-interval limit and
following the arguments of \cite{Z1} word for word, we find that, away from $\lambda=0$, the spectra
of the linearized operator about $\hat U^X(x) :=\hat U^\eps (\eps x)$ on $\Re \lambda \geq 0$ approaches, 
as $\eps = \frac{1}{X} \to 0^+$, the direct sum of the spectra of the viscous shock $\bar U$ as a solution on the whole line
plus the spectra of the constant boundary layers on the half-lines $(0,+\infty)$ and $(-\infty,1)$ 
determined by the values of $\hat U^\eps$ at $0$ and $1$ with the boundary conditions for
the steady problem at $x=0$ and $x=1$.
As the latter constant layers have been shown to be spectrally stable \cite{GMWZ}, this implies that
the spectra of $\hat U^X$ converges away from $\lambda=0$ to that of $\bar U$ as $X\to \infty$.
Rescaling, we find that, outside $B(0,c\eps^{-1})$, any $c>0$, 
the spectra of $\hat U^\eps$ are well-approximately by $\eps^{-1}$ times the spectra of $\bar U$.
We record this observation as the following proposition.
%cited in Section \ref{s:ceg}.

\begin{proposition}[Spectral decomposition]\label{sprop1}
For viscous $n$-shock solutions $\bar U$ of systems \eqref{conservation_law} satisfying (\textbf{H0})--(\textbf{H4}), 
the corresponding standing-shock family $\bar U^\eps$ contains no spectra $\Re \lambda \geq 0$ 
%outside a ball $B(0,c\eps^{-1}$ for $\eps>0$ sufficiently small, for any choice of $c>0$,
outside a ball $B(0,c\eps^{-1})$ for $\eps>0$ sufficiently small, for any choice of $c>0$, % bb changed: added ) 
if and only if $\bar U$ is spectrally stable, i.e., has no spectra $\Re \lambda \geq 0$ with $\lambda \neq 0$.
In particular, if $\bar U$ is spectrally unstable, then $\hat U^\eps$ is spectrally unstable for $\eps$ sufficiently
small.
\end{proposition}

Proposition \ref{sprop1} gives no information about the corresponding stability index and uniqueness or nonuniqueness.
However, this is provided definitively by the following result.

\begin{proposition}[Nonvanishing of the stability index]\label{sprop2}
For viscous $n$-shock solutions $\bar U$ of systems \eqref{conservation_law} satisfying (\textbf{H0})--(\textbf{H4}), 
the Evans function $D^\eps$ associated with the corresponding standing-shock family $\bar U^\eps$ satisfies 
$D^\eps(0)\neq 0$ for $\eps>0$ sufficiently small.
\end{proposition}

\begin{proof}
We only sketch the proof, which belongs more to the circle of ideas in \cite{Z1} than those of the present paper.
We first write the eigenvalue system in ``flux'' variables $(u_{II}, F)$ as
\ba\label{fluxeval}
U_{II}'&= B_{22}(U)^{-1} (F_{II}+A_{11}U_{I}), \\
F'&=\lambda A^0 U,
\ea
where $F:=B(\hat U)U'-A U$ and $U_I= A_{11}^{-1}(F_I + A_{12}U_{II})$.
This yields for $\lambda=0$ in the second equation the simple dynamics $F\equiv \const$.

Next, we observe that the Evans function may be written equivalently as
$$
D^\eps(\lambda)=
\det \bp U_{II}^1 & \dots & U_{II}^r & 0 \\
F^1& \dots & F^r & I_n\ep|_{x=1},
$$
where $\bp U_{II}^j \\F^j\ep$, $j=1, \dots, r$ denote the solutions of \eqref{fluxeval} with initial conditions
\be\label{fluxini}
\bp U_{II}^1 & \dots & U_{II}^r  \\
F^1& \dots & F^r \ep(0)=
\bp 0 \\ B(0)\bp 0 \\ I_r\ep \ep
\ee
at $x=0$.  In turn, we may view this as a Wronskian
\be\label{fullevans}
D^\eps(\lambda)=
\det \bp U_{II}^1 & \dots & U_{II}^r & U_{II}^{r+1} & \dots & U_{II}^{n+r}  \\
F^1& \dots & F^r & F^{r+1} & \dots & F^{n+r} \ep|_{x=1},
\ee
where $\bp U_{II}^j \\F^j\ep$, $j=r+1, \dots,n+ r$ denote the solutions of \eqref{fluxeval} with initial conditions
\be\label{auxini}
\bp U_{II}^{r+1} & \dots & U_{II}^{r+n}  \\
F^{r+1}& \dots & F^{r+n} \ep(1)=
\bp 0 \\ I_n \ep
\ee
at $x=1$.

By Abel's theorem, vanishing or nonvanishing of the Wronskian \eqref{fullevans} at $x=1$ is determined by
vanishing or nonvanishing at any $x\in [0,1]$.
By the analysis of \cite{Z1}, we find that, at $x=c$ for any $c>0$ sufficiently small, the solutions
$\bp U_{II}^j \\F^j\ep$, $j=1, \dots, r$ originating from $x=0$ converge exponentially in $X:=1/\eps$
to the limiting subspace of solutions of \eqref{fluxeval} on the whole line decaying at $x=-\infty$,
which may be identified by the property $F\equiv 0$, hence also $\det (U_{II}^1, \dots, U_{II}^r)\neq 0$.
Recalling by the simple dynamics for $\lambda=0$ that 
$$
(F^{r+1}, \dots, F^{r+n})\equiv 
\bp  B(0)\bp 0 \\ I_r\ep \ep,
$$
we find that the Wronskian at $x=c$ converges exponentially in $X=\eps^{-1}$ to 
$$
\det (U_{II}^1, \dots, U_{II}^r)|_{x=c}\neq 0,
$$
hence $D^\eps(0)\neq 0$ for $\eps>0$ sufficiently small.
For further details, see \cite{Z1}.
\end{proof}

\br\label{standrmk}
The result of Proposition \ref{sprop2}, though proved by similar techniques, stands in striking contrast to
the results of \cite{Z1,SZ} in the half-line case, where the stability index was seen to change sign as parameters were
varied for (full) polytropic gas dynamics, despite stability of the underlying whole-line shock profile.
The standing-shock construction can nonetheless be useful in seeking instability/bifurcation,
however, through bifurcation from unstable to stable background shock; this strategy is used successfully,
e.g., in \cite{paper1}.
\er

\subsection{Large-viscosity/small interval asymptotics}\label{s:large_visc}
It seems interesting to ask also what are the asymptotics of solutions
as viscosity goes to infinity instead of zero, or,  equivalently, interval length goes to zero with viscosity fixed. Before studying the general case (under reasonable assumptions) we provides two examples that illustrate what can happen for such an asymptotic.

\subsubsection{The scalar case}\label{s:large_scalar}

We consider here the scalar case as in Section \ref{s:asymptoticI}. For a $\mathcal{C}^{1}$ function $f$, $\nu \geq 1$ and $u_{0},u_{1} \in \R$, we consider the unique solution $u^{\nu}$ (as proved in Proposition \ref{existence_steady}) of
\begin{equation}\label{steady_scalar_2}
\left\{
\begin{array}{l}
\nu \hat u^{\nu}_{xx} =(f(\hat u^{\nu}))_{x} \text{  ,   } 0<x<1,\\
\hat u^{\nu}(0)=u_{0} \text{  ,  } \hat u^{\nu}(1)=u_{1}.
\end{array}
\right.
\end{equation}
The following proposition provides the behavior of $\hat u^{\nu}$ as $\nu \to +\infty$.

\bpr
Let $\nu \geq 1$ and $u_{0},u_{1} \in \R$ and assume that $f$ is $\mathcal{C}^{1}(\R)$. Consider $\hat u^{\nu}$ the unique solution of \eqref{steady_scalar_2}. There exists a constant $C>0$ depending only on $f$, $u_{0}$ and $u_{1}$ such that for any $x \in [0,1]$,
\[
\left|\hat u^{\nu}(x) - \left( (u_{1}-u_{0})x + u_{0} \right) \right| \leq \frac{C}{\nu}.
\]
\epr

\begin{proof}
Integrating the equation from $0$ to $x \in [0,1]$, there exists a constant $C_{1}>0$ depending only on $f$, $u_{0}$ and $u_{1}$ so that
\[
|\hat u_{x}^{\nu}(x)-u_{x}^{\nu}(0)| \leq \frac{C_{1}}{\nu}.
\]
Since $\min(u_{0},u_{1}) \leq u^{\nu} \leq \max(u_{0},u_{1})$, $\hat u_{x}^{\nu}$ is therefore necessarily uniformly bounded with respect to $\nu \geq 1$. Thus $\nu \hat u_{xx}^{\nu}$ is uniformly bounded with respect to $\nu \geq 1$. Integrating two times, we obtain that the quantity $\nu (u_{1}-u_{0}- \hat u_{x}^{\nu}(0))$ is uniformly bounded with respect to $\nu \geq 1$. The desired estimate follows.
\end{proof}

\subsubsection{Full gas dynamics}\label{s:large_full}
For full gas dynamics \cite[Eq. (2.6)]{paper1}, \eqref{steady_conservation_law} becomes
\ba\label{prof}
\frac{\alpha}{u_0} u' &= c_1 + u + \Gamma \frac{e}{u},
\qquad
\frac{\nu}{u_0}e'= c_2 -c_1 u - \frac{1}{2} u^2 + e,
\ea
where $u>0$ is fluid velocity, $e>0$ specific internal energy, and
\be\label{link_c_derivative_at_0}
c_{1}=\frac{\alpha}{u_{0}} u'(0) - u_{0} - \Gamma \frac{e_{0}}{u_{0}} \text{ , } \quad
c_{2}=\frac{\nu}{u_{0}} e'(0) + \alpha u'(0) - e_{0} - \frac{1}{2} u_{0}^{2} - \Gamma e_{0}.
\ee
(Here, we have taken without loss of generality $\rho_0, u_0=1$, so that by the steady mass conservation
equation, $\rho u\equiv 1$, where $\rho$ is fluid density \cite{paper1}.)

Formally taking $\alpha$, $\nu$ to infinity, and dropping $O(1)$ terms, we obtain limiting equations
$\frac{\alpha}{u_0} u' = c_1$, $\frac{\nu}{u_0}e'= c_2 -c_1 u$ 
with
$c_{1}=\frac{\alpha}{u_{0}} u'(0)$ , $c_{2}=\frac{\nu}{u_{0}} e'(0)$, or
\ba\label{fprof}
\bar u' = u'(0), \qquad \frac{1}{u_0}\bar e'= \bar e'(0)  -\frac{\alpha}{\nu}\bar u'(0)  \bar u, 
\ea
giving exact solution
\be\label{lexact}
\bar u(x)=  u_0 + x(u_1-u_0), \;
\bar e(x)=x \Big( e_1- \frac{\alpha}{\nu}\frac{ u_1^2  - u_0^2}{2}\Big)
+ \frac{\alpha}{\nu}(u_1-u_0)\big( u_0 x + (u_1-u_0)x^2/2\big).
\ee
%NOTES:
%with $e_1=e(1)= e'(0)+ \frac{\alpha}{\nu}\frac{ u_1^2  - u_0^2}{2}$
%with $e'(0)=e_1- \frac{\alpha}{\nu}\frac{ u_1^2  - u_0^2}{2}$.

For $\eps>0$, set $E_{\eps} = \{(u,e) \in \mathbb{R}^2 \text{ , } \eps < u , e < \frac{1}{\eps} \}$.

\begin{proposition}\label{Lgas}
	For any fixed $\eps>0$ and profile \eqref{lexact} contained in $E_\eps$, 
	holding $r =\nu\alpha>0$ fixed, and taking $\alpha$, $\nu$ to infinity,
	there is for sufficiently large $(\alpha, \nu)$ a unique steady profile of \eqref{prof}
	lying in $E_\eps$, converging to the formal limit \eqref{lexact}
	in $H^1[0,1]$ at rate $O(\alpha^{-1})$.
\end{proposition}

\begin{proof}
	Put back into second-order form, \eqref{prof} becomes (cf. \cite[Eq. (2.5)]{paper1})
\ba\label{ellipt}
	\Big(\frac{\alpha}{u_0} u'\Big)' &= (u + \Gamma \frac{e}{u})',
\qquad
	\Big(\frac{\nu}{u_0}e' + \frac{\alpha}{u_0}uu'   \Big)'=  (\Gamma e + e +  \frac{1}{2} u^2)',
\ea

	Writing $u=\bar u+ w$, $e=\bar e+z$, we have, substituting in \eqref{ellipt}, 
	substracting the corresponding equation for $\bar u$, dividing by $\alpha$, the elliptic perturbation
	equations
	\be\label{wz}
	(\frac{1}{u_0} w')' = (O(\alpha^{-1}))',
\qquad
	(\frac{r}{\alpha u_0}e'+ \frac{1}{u_{0}}(\bar u w' + w \bar u' + w w')'=  +  (O(\alpha^{-1}))',
\ee
	with homogeneous data $(w,z)(0)=(0,0)$, $(w,z)(1)=(0,0)$, where estimate $O(\alpha^{-1})$
	(by Sobolev embedding) remains valid so long as $\|w\|_{H^1[0,1]}$ is sufficiently small.

	With this structure, it is straightforward to carry out a contraction mapping argument for 
	$w$ in a sufficiently small ball $B$ in $H^1[0,1]$, considering the right-hand side as input and the
	solution of the block-triangular operator on the left-hand side as image, and showing by energy
	estimates that the resulting operator is contractive on $B$ with contraction constant $O(\alpha^{-1})<1$,
	giving existence and uniqueness by the Banach fixed point theorem of a solution in $B$, i.e., 
	small in $H^1[0,1]$.  Thus, the righthand side of \eqref{wz} is $O(\alpha^{-1})$, and,
	solving, we have $\|(w,z)\|_{H^1[0,1]}=O(\alpha^{-1})$ as well.
	Recalling that $(w,z)$ by definition is the difference between exact and limiting solutions, we are done.
\end{proof}

\br\label{isormk}
The large-viscosity analysis for isentropic gas dynamis is similar but much simpler,
yielding $\bar u$ as the line interpolating between $u(0)$ and $u(1)$.
This is to be compared with the rich small-viscosity behavior depicted in Figure \ref{steadysolisentropic}.
\er

\subsubsection{General case}\label{s:large_gen}
Now, let us consider the general case, under assumption \eqref{solve} and, possibly after a coordinate change,
$B_{22}$ symmetric positive definite: this is possible for example for systems with convex entropy \cite{KSh}.
Adding a variable viscosity $\nu$ in \eqref{conservation_law} gives
	$$
	\partial_{t}f^0(U) + f(U)_{x} = \nu \left(B(U) U_{x} \right)_{x},
	$$
	leading to the steady problem
\be\label{Lode_general}
\nu B_{22}(U) U_{II}' = f_{II}(U) - f_{II}(U_{0}) + \nu \tilde C_2 \text{ , } \quad 
U_{II}(0)=U_{0II} \text{  ,  } f_{I}(U)=f_{I}(U_{0}),
\ee
where $\tilde C_2= B_{22}(U_{0}) C_{2}= B_{22}U_{II}'(0)$, and functions of $U$ are viewed as functions
on $U_{II}$ alone, through the dependence of $U_I$ on $U_{II}$ imposed by relation \eqref{solve}.

As in the previous case, we expect that solutions will converge to the solution of the formal limiting equations
obtained by multiplying by $\nu^{-1}B_{22}(\bar U_{II})^{-1}$ and taking $\nu\to \infty$, or
\be\label{Lformal}
 \bar U_{II}' =  B_{22}(\bar U_{II})^{-1} \tilde C_2 \text{ , }  \quad
 \bar U_{II}(0)=U_{0II} \text{  ,  } f_{I}(\bar U)=f_{I}(U_{0}).
\ee
A first question, partially answered in the next lemma, is whether \eqref{Lformal} admits a solution.

\begin{lemma}\label{formexist}
	Let $0<\beta_0 < B_{22}^{-1} < \beta_1$ on a subset $\tilde{ \mathcal{U}}$ of the domain of definition $\mathcal{U}$ of \eqref{Lode_general},
	and set $\theta= \cos^{-1}(\beta_0/\beta_1)$.
	If the truncated cone $\TT$ of angle $\theta$ based at $U_{0II}$ and centered about the segment
	from $U_{0II}$ to $U_{0II} +(\beta_1/\beta_0)(U_{1II}-U_{0II})$
	is contained in $\tilde{\mathcal{U}}$, then (i) any steady solution $\hat U$
	valued in $\tilde{\mathcal{U}}$ of the formal limiting equation \eqref{Lformal} lies also in $\TT$,
	and (ii) there is at least one steady solution $\hat U$ lying in $\TT$, hence valued in 
	$\tilde{\mathcal{U}}$.
\end{lemma}

\begin{proof}
	By $\beta_0<B_{22}^{-1}<\beta_1$, for any vector $V$ the	
angle $\theta$ between $V$ and $B_{22}^{-1}V$ satisfies
$$\cos\theta =\frac{ \langle V, B_{22}^{-1} V\rangle}{ |V||B_{22}^{-1}V| }\geq \frac {\beta_0}{\beta_1},$$
	with length of $B_{22}^{-1} V$ lying between $\beta_0 |V|$ and $\beta_1 |V|$.

	Likewise, so long as $U_{II}$ remains in $\tilde{ \mathcal{U}}$, for $0\leq x\leq 1$
	the vector 
	$$
	U_{II}(x)-U_{II}(0)=\Big(\int_0^x B_{22}^{-1}(U_{II}(y))dy\Big) \tilde C_2= B_{ave} (x \tilde C_2),
	$$
	by the corresponding bounds $\beta_0<B_{ave}<\beta_1$,
	lies within the same angle $\theta$ of $\tilde C_2$, with length between $x\beta_0 |\tilde C_2|$
	and $x\beta_1 |\tilde C_2|$.
	If $U_{II}$ is a solution of \eqref{Lode_general}, i.e., $U_{II}(0)=U_{0II}$ and $U_{II}(1)=U_{122}$,
	then this implies that $\tilde C_2$ is within angle $\theta$ of $U_{1II}-U_{0II}$, and thus $U_{II}(x)$,
	lying within angle $\theta$ of $\tilde C_2$, lies in the truncated cone $\TT$.

	On the other hand, for $\tilde C_2$ lying within the narrower truncated cone $\TT_2$ of angle $\theta$ 
	around $V=U_{1II}-U_{0II}$, and length between $|V|/\beta_1$ and $|V|/\beta_2$, we have by the same 
	estimates that $U_{II}$ stays within the interior of $\TT$ for $0\leq x\leq 1$,
	hence $\Psi(\tilde C_2)=U_{II}(1)-U_{0II}$ is well-defined.
	Varying $C_2$ around the boundary of $\TT_2$, we find by the same angle estimate that the 
	degree of $\Psi$ about $0$ is $+1$, giving existence of at least one steady profile in $\TT$.
\end{proof}

\br
By a constant change of coordinates, we may take
$$
B_{22}(U_{II})\to \tilde B_{22}(U_{II})= 
	B_{22}(U_{0II})^{-1/2} B_{22}(U_{II}) B_{22}(U_{0II})^{-1/2},
$$
in particular $\tilde B_{22}(U_{0II})=\Id$, to improve the condition number $\beta_1/\beta_0$.
For $B_{22}\equiv \const$, as for gas dynamics, this yields $\tilde B_{22}\equiv \Id$, 
giving an exact solution $U_{II}(1)=\tilde B_{22}^{-1}\tilde C_2$.
In general, we do not see that there is necessarily a solution of \eqref{Lode_general} lying
in $\tilde {\mathcal{U}}$, even ``nice'' domains arising in applications,
 nor that solutions of \eqref{Lode_general} must necessarily be unique.
\er

Mimicking our treatment of the gas dynamical case, we may write \eqref{Lode_general} in elliptic form 
%to get existence, convergence of solution to full model \eqref{Lode_general} for $\nu$ sufficiently large.
\be\label{eLode_general}
(\nu B_{22}(U) U_{II}')' = (f_{II}(U))' \text{ , } \quad 
U_{II}(0)=U_{0II} \text{  ,  } f_{I}(U)=f_{I}(U_{0}),
\ee
then, defining $W=\hat U_{II}- \bar U_{II}$ to be the difference between exact and limiting solutions,
rewrite as an elliptic perturbation system
\be\label{eLode_pert}
( B_{22}(\bar U) W'
+ (dB_{22}(\bar U)W) \bar U_{II}')' = 
 ((dB_{22}(\bar U)W) W')'+ (O(\nu^{-1}) )',
\ee
with Dirichlet boundary conditions $W(0)=W(1)=0$.
If, as in the gas-dynamical case, the operator on the lefthand side is uniformly elliptic, then
obtain an existence/convergence result as in Proposition \ref{Lgas}.
by the same sort of contraction mapping argument used there.
However, we do not see in general why this should be true, {\it except for small data 
$|\bar U'|\ll 1$.}

%\appendix

%%%%%%%%%%%%%%%%%%%%%%%%%%%%%%%%%%%%%%%%%%%%%%%%%%%%%%%%%%%%%%%%%%%%%%%%%%%%%%%%%%%%%%%%%%%%%%%%%%%%%%%%%
%%%%%%%%%%%%%%%%%%%%%%%%%%%%%%%%%%%%%%%%%%%%%%%%%%%%%%%%%%%%%%%%%%%%%%%%%%%%%%%%%%%%%%%%%%%%%%%%%%%%%%%%%
%\newpage
\scriptsize
\bibliographystyle{alpha}
\bibliography{biblio}
\normalsize
\end{document}